\documentclass[EJP,preprint]{ejpecp}

\SHORTTITLE{Equilibrium perturbations}

\TITLE{Equilibrium perturbations for stochastic interacting systems\support{Supported by the ANR grant MICMOV (ANR-19-CE40-0012)
of the French National Research Agency (ANR) and the Fundamental Research Funds for the Central Universities in China.}\
    } 

\AUTHORS{%
  Lu~Xu\footnote{Gran Sasso Science Institute, 67100 L'Aquila, Italy.
    \EMAIL{lu.xu@gssi.it}}
  \and 
  Linjie~Zhao\footnote{School of Mathematics and Statistics, Huazhong University of Science \& Technology, Wuhan, China.
  \EMAIL{linjie_zhao@hust.edu.cn}}}


\AMSSUBJ{60K35 ; 82C22} 

\SUBMITTED{June 5, 2022} 
\ACCEPTED{December 29, 2022} 

\VOLUME{0}
\YEAR{2020}
\PAPERNUM{0}
\DOI{10.1214/YY-TN}

\ABSTRACT{We consider the equilibrium perturbations for two stochastic systems: the $d$-dimensional generalized exclusion process and the one-dimensional chain of anharmonic oscillators. We add a perturbation of order $N^{-\alpha}$ to the equilibrium profile and speed up the process by $N^{1+\kappa}$ for parameters $0<\kappa\le\alpha$. Under some additional constraints on $\kappa$ and $\alpha$, we show the perturbed quantities evolve according to the Burgers equation in the exclusion process, and to two decoupled Burgers equations in the anharmonic chain, both in the smooth regime.}

\newcommand \bN {\mathbb N_+}
\newcommand \bR {\mathbb R}
\newcommand \bT {\mathbb T}
\newcommand \bZ {\mathbb Z}

\newcommand \cA {\mathcal A}

\newcommand \cE {\mathcal E}
\newcommand \cL {\mathcal L}
\newcommand \cR {\mathcal R}
\newcommand \cS {\mathcal S}
\newcommand \cY {\mathcal Y}

\newcommand \fp {\mathfrak p}
\newcommand \fr {\mathfrak r}

\newcommand \ep {\epsilon}
\newcommand \ve {\varepsilon}

\newcommand \la {\lambda}	
\newcommand \vf {\varphi}
\newcommand \vr {\varrho}

\newcommand \bsf {\boldsymbol\phi}
\newcommand \bst {\boldsymbol\tau}
\newcommand \bsu {\boldsymbol\upsilon}

\newcommand{\Romannum}[1]{\uppercase\expandafter{\romannumeral#1\relax}}
\newcommand{\romannum}[1]{\romannumeral#1\relax}

\newcommand \CS {Cauchy--Schwarz }
\newcommand \Gro {Gr\"onwall's }

\newcommand \ini {\mathrm{ini}}

\renewcommand{\>}{\big\rangle}

\newcommand{\R}{\mathbb R}
\newcommand{\Z}{\mathbb Z}

\newcommand{\T}{\mathbb T}

\newcommand{\Lcal}{\mathcal{L}} 
\newcommand{\Ecal}{\mathcal{E}}

\newcommand{\qfrak}{\mathfrak{q}}
\newcommand{\pfrak}{\mathfrak{p}}

\begin{document}



\section{Introduction}
\label{sec:intr}

One of the main aims in the theory of hydrodynamic limit is to derive partial differential equations from microscopic systems. For asymmetric systems with one or several conservation laws, the conserved quantity/quantities usually evolve macroscopically according to the hyperbolic system under the Euler scaling (time accelerated by $N$ and space divided by $N$), \emph{cf.}\,\cite{rezakhanlou91,OVY93,klscaling,BHO19} for example.  To refine the hydrodynamic limit, when the system has only one conservation law, in the seminal paper \cite{esposito1994diffusive}, Esposito, Marra and Yau perturb the asymmetric simple exclusion processes around the equilibrium point with order $N^{-1}$ at the initial time in dimension $d \geq 3$, and show that the evolution of the perturbed quantity, namely the \emph{equilibrium perturbation}, is governed by the viscous Burgers equation under the diffusive scaling (time accelerated by $N^2$ and space divided by $N$). This is closely related to the understanding of Navier--Stokes corrections and is called the incompressible limit in the literature, \emph{cf.}\,\cite[Section 7.7]{klscaling} for details. Later, such kind of result  is extended by Sepp{\"a}l{\" a}inen \cite{seppalainen2001perturbation} to dimension $d=1$. He considers the one-dimensional Hammersley’s model, adds a perturbation of order $N^{-\alpha}$ to the equilibrium,  and shows that the perturbation macroscopically obeys the inviscid Burgers equation in the time scale $N^{1+\alpha}t$ if $0 < \alpha< 1/2$. The proof uses combinatorial properties of the Hammersley's model and coupling techniques, and the result holds even beyond the appearance of shocks. Independently, T{\'o}th and Valk{\'o} \cite{toth2002between} obtain the same result for a large class of one-dimensional interacting particle systems, which is called \emph{deposition} models in their paper, but only for  $0 < \alpha < 1/5$. T{\'o}th and Valk{\'o} employ Yau's relative entropy method \cite{yau1991relative}, thus the result holds only in the smooth regime of the solutions. Very recently, Jara, Landim and Tsunoda \cite{JLT21} consider equilibrium perturbations in weakly asymmetric exclusion processes and derive viscous Burgers equation under the diffusive scaling.
Equilibrium perturbations have also been considered for systems with two conservation laws. In \cite{toth2005perturbation}, T{\'o}th and Valk{\'o} prove for a very rich class of systems that the small perturbations are universally driven by a two-by-two system. In \cite{Valko06}, Valk{\'o} shows that small perturbations around a hyperbolic equilibrium point evolve according to two decoupled Burgers equations.
For system with three or more conservation laws, the problem remains generally open.
 
\subsection{Nonlinear geometric optics for conservation laws}

We start from illustrating the idea of nonlinear geometric optics.
For $n\ge1$, consider an $n$-system of conservation laws
\begin{equation}
\label{eq:cl}
  \partial_t\mathbf f(t,v)+\big(m\cdot\nabla_v\big)J\big(\mathbf f(t,v)\big)=0, \quad \mathbf f(0,\cdot)=\mathbf f^\ini,
\end{equation}
where $v\in\bR^d$, $\mathbf f^\ini \in C^1(\bR^d;\bR^n)$, $m\in\bR^d$, $J:\bR^n\to\bR^n$ is a smooth, nice function and
\begin{equation*}
  \big(m\cdot\nabla_v\big)J\big(\mathbf f(v)\big) := \left( m_1\frac \partial{\partial v_1}+...+m_d\frac\partial{\partial v_d} \right) J\big(\mathbf f(v)\big) \in \bR^n.
\end{equation*}
To avoid the discussion on boundary terms, we adopt the periodic condition $\mathbf f(t,v+e_i)=\mathbf f(t,v)$ for each $i=1$, ..., $d$.
It is equivalent to say that the space variable $v\in\bT^d:=(\bR/\bZ)^d$.

The stationary solutions to \eqref{eq:cl} are given by constant functions.
We are interested in the behavior of the solution when the initial condition is slightly perturbed from the stationary.
More precisely, fix a vector $\mathbf w_*\in\bR^n$ and consider the perturbed system
\begin{equation*}
  \partial_t\mathbf f^\ve(t,v)+\big(m\cdot\nabla_v\big)J\big(\mathbf f^\ve(t,v)\big)=0, \quad \mathbf f^\ve(0,\cdot)=\mathbf w_*+\ve\mathbf w^\ini,
\end{equation*}
where $\ve>0$ and $\mathbf w^\ini \in C^1(\bT^d;\bR^n)$.
Suppose that $\mathbf f^\ve$ decomposes as
\begin{equation*}
  \mathbf f^\ve(t,v)=\mathbf w_*+\ve\mathbf w(t,v)+\mathbf w^\ve(t,v), \quad |\mathbf w^\ve(t,v)|=o(\ve).
\end{equation*}
By expanding $J$ at $\mathbf w_*$, we obtain the linearized equation of the first order component:
\begin{equation}
\label{eq:linear-cl}
  \frac{\partial\mathbf w}{\partial t} + A\sum_{i=1}^d m_i\frac{\partial\mathbf w}{\partial v_i}=0, \quad \mathbf w(0,\cdot)=\mathbf w^\ini,
\end{equation}
where $A=DJ(\mathbf w_*)$ is the Jacobian matrix.
Assume that \eqref{eq:cl} is strictly hyperbolic at $\mathbf w_*$: $A$ has $n$ distinct, non-zero eigenvalues $\la_1$, ..., $\la_n$, with the corresponding left (right) eigenvectors denoted by $\mathbf l_1$, ..., $\mathbf l_n$ ($\mathbf r_1$, ..., $\mathbf r_n$):
\begin{align*}
  \mathbf l'_j(\la_j\mathrm{Id}-A)=(\la_j\mathrm{Id}-A)\mathbf r_j=0, \quad \mathbf l'_j\mathbf r_k=\mathbf 1_{\{j=k\}}.
\end{align*}
From \eqref{eq:linear-cl}, $(d/dt)\mathbf l'_j\mathbf w=0$ along each characteristic lines $(t,v+\la_jtm)$.
Hence,
\begin{align*}
  \mathbf w(t,v)=\sum_{j=1}^n \mathbf l'_j\mathbf w^\ini\big(\phi_j(t,v)\big)\mathbf r_j, \quad \phi_j:=v-\la_jtm.
\end{align*}

To observe non-trivial evolution along these lines, we investigate at longer time $\ve^{-1}t$.
Assume that the solution decomposes further to the second order as
\begin{align*}
  \mathbf f^\ve(t,x)=\mathbf w_*+\ve\sum_{j=1}^n \sigma_j(\ve t,\phi_j)\mathbf r_j+\ve^2\sum_{j=1}^n \tilde\sigma_j(\ve t,\bsf)\mathbf r_j+o(\ve^2),
\end{align*}
where $\bsf=(\phi_1,\ldots,\phi_n)\in\bR^{d \times n}$, $\sigma_j=\sigma_j(s,u)$ and $\tilde\sigma_j=\tilde\sigma_j(s,u_1,\ldots,u_n)$ are $C^1$ functions for $j=1$, ..., $n$.
Inserting the expansion into \eqref{eq:cl} and noting that $A\mathbf r_j=\la_j\mathbf r_j$, we see that the terms of order $\ve$ cancel autonomously, while the higher order terms read
\begin{equation*}
  \begin{aligned}
    \ve^2\sum_{j=1}^n \partial_s\sigma_j(\ve t,\phi_j)\mathbf r_j+\ve^2\sum_{j,j'} \big((\la_j-\la_{j'})m\cdot\nabla_{u_{j'}}\big)\tilde\sigma_j(\ve t,\bsf)\mathbf r_j\\
    +\ve^2\sum_{j,j'} \sigma_j(\ve t,\phi_j)\big(m\cdot\nabla_u\big)\sigma_{j'}(\ve t,\phi_{j'})H(\mathbf r_j,\mathbf r_{j'})=o(\ve^2),
  \end{aligned}
\end{equation*}
where $H$ refers to the Hessian matrix of $J$. Since $\mathbf l'_k\mathbf r_j=\mathbf1_{\{k=j\}}$, the above holds true if and only if for each $k=1$, ..., $n$,
\begin{equation}
\label{eq:burgers0}
  \begin{aligned}
    \partial_s\sigma_k(s,u_k)+\sum_{j=1}^n \big((\la_k-\la_j)m\cdot\nabla_{u_j}\big)\tilde\sigma_k(s,u_1,\ldots,u_n)\\
    +\sum_{j,j'} \mathbf l'_kH(\mathbf r_j,\mathbf r_{j'})\sigma_j(s,u_j)\big(m\cdot\nabla_u\big)\sigma_{j'}(s,u_{j'})=0.
  \end{aligned}
\end{equation}

We restrict our discussion to the \emph{non-resonant} situation, in which the wave of each frequency $\la_k$ performs self-consistent time evolution governed by the Burgers equation
\begin{equation}
\label{eq:burgers}
  \partial_s\sigma_k + \mathbf l'_kH(\mathbf r_k,\mathbf r_k)\big(m\cdot\nabla_u\big)\big(2^{-1}\sigma_k^2\big)=0, \quad \forall\,k=1,\ldots,n.
\end{equation}
Comparing \eqref{eq:burgers0} and \eqref{eq:burgers}, $\tilde\sigma_k$ then has to solve the equation
\begin{equation*}
  \begin{aligned}
    &\sum_{j=1}^n \big((\la_j-\la_k)m\cdot\nabla_{u_j}\big)\tilde\sigma_k(s,u_1,\ldots,u_n)\\
    =&\sum_{(j,j')\not=(k,k)} \mathbf l'_kH(\mathbf r_j,\mathbf r_{j'})\sigma_j(s,u_j)\big(m\cdot\nabla_u\big)\sigma_{j'}(s,u_{j'}).
  \end{aligned}
\end{equation*}
By superposition, we only need to construct $\tilde\sigma_k$ as
\begin{equation*}
  \tilde\sigma_k(s,u_1,\ldots,u_n) := \sum_{j\not=k} \frac{\mathbf l'_kH(\mathbf r_j,\mathbf r_j)\sigma_j^2(s,u_j)}{2(\la_j-\la_k)}+\sum_{j\not=j'} \tilde\sigma_{k,j,j'}(s,u_j,u_{j'}),
\end{equation*}
where the functions $\tilde\sigma_{k,j,j'}=\tilde\sigma_{k,j,j'}(s,u,u')$ satisfy that
\begin{equation}
\label{eq:correction0}
  \begin{aligned}
    \sum_{j=1}^n \big((\la_j-\la_k)m\cdot\nabla_u\big)\sum_{j\not=j'} \tilde\sigma_{k,j,j'}(s,u_j,u_{j'})\\
    =\sum_{j\not=j'} \mathbf l'_kH(\mathbf r_j,\mathbf r_{j'})\sigma_j(s,u_j)\big(m\cdot\nabla_u\big)\sigma_{j'}(s,u_{j'}).
  \end{aligned}
\end{equation}
Hence, we formally obtain a \emph{sufficient condition} for the non-resonant case:
\begin{equation}
\label{eq:correction}
  \begin{aligned}
    &\big((\la_j-\la_k)m\cdot\nabla_u+(\la_{j'}-\la_k)m\cdot\nabla_{u'}\big)\tilde\sigma_{k,j,j'}(s,u,u')\\
    =\,&\mathbf l'_kH(\mathbf r_j,\mathbf r_{j'})\sigma_j(s,u)\big(m\cdot\nabla_u\big)\sigma_{j'}(s,u'),
  \end{aligned}
\end{equation}
must be solvable for all $k$, $j$, $j'=1$, ..., $n$ such that $j\not=j'$.
Notice that the formal calculation above apparently relies on the smoothness of $\sigma_k$, so we need to assume that $t<T_\mathrm{shock}$, where $T_\mathrm{shock}$ is the time when the shock first appears in the entropy solution to \eqref{eq:burgers}.

Observe that the non-resonant condition holds autonomously if $n=1$.
If $n=2$, $d=1$, without loss of generality we can set $m=1$ and solve \eqref{eq:correction} explicitly as
\begin{equation*}
  \begin{aligned}
    \tilde\sigma_{1,1,2}=c_1\sigma_1(s,u)\sigma_2(s,u'), \quad
    \tilde\sigma_{1,2,1}=c_1\Sigma_2(s,u)\partial_u\sigma_1(s,u'),\\
    \tilde\sigma_{2,1,2}=c_2\Sigma_1(s,u)\partial_u\sigma_2(s,u'), \quad
    \tilde\sigma_{2,2,1}=c_2\sigma_2(s,u)\sigma_1(s,u'),
  \end{aligned}
\end{equation*}
where $c_1=-(\la_1-\la_2)^{-1}\mathbf l_1'H(\mathbf r_1,\mathbf r_2)$, $c_2=(\la_1-\la_2)^{-1}\mathbf l_2'H(\mathbf r_1,\mathbf r_2)$ and $\Sigma_1$, $\Sigma_2$ are primitives of $\sigma_1$, $\sigma_2$, respectively, i.e., $\Sigma_1=\int_0^\cdot \sigma_1du$, $\Sigma_2=\int_0^\cdot \sigma_2du$.
To make $\Sigma_j$ well-defined on $\bT$, we shall assume in addition that $\int_\bT \sigma_j(0,u)du=0$ for both $j=1$, $2$, or equivalently
\begin{equation}
\label{eq:assp-chain}
  \int_\bT \mathbf w^\ini(u)du=\int_\bT \big[\sigma_1(0,u)\mathbf r_1+\sigma_2(0,u)\mathbf r_2\big]du=0.
\end{equation}
The non-resonant expansion is not in general valid for the system with $n\ge3$.
We refer the readers to \cite{Majda84,DiPerM85} for rigorous justification and detailed discussion for the resonant case.

\subsection{Equilibrium perturbation}

The goal of the present paper is to derive the non-resonant system \eqref{eq:burgers} of Burgers equations as a decent scaling limit for some stochastic interacting system.
Formally speaking, for a scaling parameter $N\in\bN$, we study a particle system $\zeta(t)=\{\zeta_x(t)\}$, where the microscopic position $x$ belongs to the periodic lattice $\bT_N^d:=\bZ^d/(N\bZ^d)$.
Suppose that the system has $n$ conserved quantities $\mathbf g=(g_1,\ldots,g_n)$ such that, under the \emph{Euler} space-time rescaling $(Nt,[Nv])$ their empirical densities evolve with the macroscopic conservative system (\emph{hydrodynamic limit}):
\begin{equation*}
  \lim_{N\to\infty} \frac1{N^d}\sum_{x\in\bT_N^d} \mathbf g\big(\zeta_x(Nt)\big)\delta_{\frac xN}(dv) \Rightarrow \mathbf f(t,v)dv, \quad N\to\infty,
\end{equation*}
where $\mathbf f$ is the solution to \eqref{eq:cl}.
To do the perturbation, fix $\alpha>0$, $\mathbf w_*\in\bR^n$, $\mathbf w^\ini \in C^1(\bT^d;\bR^n)$ and start the dynamics from some initial distribution such that
\begin{equation*}
  \frac1{N^d}\sum_{x\in\bT_N^d} \mathbf g\big(\zeta_x(0)\big)\delta_{\frac xN}(dv) \approx \big(\mathbf w_*+N^{-\alpha}\mathbf w^\ini(v)\big)dv, \quad N\to\infty.
\end{equation*}
Let $\{\sigma_k;k=1,\ldots,n\}$ be the smooth solutions to \eqref{eq:burgers}.
In the non-resonant situation, the argument in the previous part suggests the formal asymptotic formula
\begin{equation*}
  \frac1{N^d}\sum_{x\in\bT_N^d} \mathbf g\big(\zeta_x(Nt)\big)\delta_{\frac xN}(dv) \approx \left[ \mathbf w_*+N^{-\alpha}\sum_{j=1}^n \sigma_j(N^{-\alpha}t,v-\la_jtm)\mathbf r_j \right] dv.
\end{equation*}
Choosing $\kappa\le\alpha$ and using the variables $s=N^{-\kappa}t$, $u=v-\la_ktm$, one obtains that
\[ \frac1{N^{d}}\sum_{x\in\bT_N^d} \mathbf l'_k\big[\mathbf g\big(\zeta_x(N^{1+\kappa}s)\big)-\mathbf w_*\big] \delta_{\frac xN-N^\kappa\la_ksm}(du) \approx N^{-\alpha} \sigma_k(N^{\kappa-\alpha}s,u)du,\]
or equivalently,
\begin{equation}
\label{eq:equi-pert0}
  \frac1{N^{d-\alpha}}\sum_{x\in\bT_N^d} \mathbf l'_k\big[\mathbf g\big(\zeta_x(N^{1+\kappa}s)\big)-\mathbf w_*\big] \delta_{\frac xN-N^\kappa\la_ksm}(du) \approx \sigma_k^{(\alpha,\kappa)}(s,u)du,
\end{equation}
as $N\to\infty$ for each $k=1$, ..., $n$, where
\begin{equation*}
  \sigma_k^{(\alpha,\kappa)}(s,u)=
  \begin{cases}
    \sigma_k(s,u), &\text{if }\,\kappa=\alpha,\\
    \sigma_k^\ini(u):=\mathbf l'_k\mathbf w^\ini(u), &\text{if }\,\kappa<\alpha.
  \end{cases}
\end{equation*}

\begin{remark}[Fluctuation]
The convergence in \eqref{eq:equi-pert0} is available only if $\alpha<d/2$.
To see that, let us assume that the equilibrium states of the dynamics are given by the family of canonical Gibbs measures $\otimes_x \nu_{\mathbf w}(d\zeta_x)$, where $\mathbf w$ is the corresponding equilibrium value of the conserved vector $\mathbf g$.
Starting the dynamics from equilibrium initial state with $\mathbf w=\mathbf w^*$.
If $\mathbf g$ possesses finite variance, then the central limit theorem for i.i.d. random variables yields that for each time $s>0$,
\begin{equation*}
  \frac1{N^{d/2}}\sum_{x\in\bT_N^d} \mathbf l'_k\big[\mathbf g\big(\zeta(s)\big)-\mathbf w^*\big]\delta_{\frac xN}(du)
\end{equation*}
converges weakly, as $N\to\infty$, to some Gaussian random field.
The macroscopic time evolutions of these fields are called the equilibrium fluctuations.
For the equilibrium fluctuations for the models studied in this paper, see, e.g., \cite{chang2001equilibrium,Goncalves08,OllaX20}.
\end{remark}

In this paper, we discuss two specific microscopic interacting systems: (\romannum1) the generalized exclusion in any dimension, where the total number of particles is the only conservation law ($n=1$) and (\romannum2) the nonlinear Hamiltonian dynamics in one dimension, disturbed by noises that preserve only the total momentum and volume stretch ($n=2$).
We derive the equilibrium perturbation in \eqref{eq:equi-pert0} as the hydrodynamic limit under the incompressible Euler scaling, see Theorem \ref{coro1} and \ref{thm:chain-equi-pert}, respectively.

\subsubsection{Generalized exclusion process}

In Section \ref{sec:gasep} and \ref{sec:gasep-rel-ent}, we study the $d$-dimensional \emph{generalized exclusion process} with at most $K>0$ particles per site. By \cite{rezakhanlou91,klscaling}, the density of the particles under the Euler scaling evolves according to the hyperbolic equation
\begin{equation}\label{hyperPDE}
\partial_t \varrho + m \cdot \nabla J(\varrho) = 0, \quad \varrho(0,\cdot)=\varrho^{\rm ini},
\end{equation}
where $J(\varrho) = \varrho (K-\varrho)$ is the macroscopic flux at density $\varrho$, and $m \neq 0$ is the mean  of the underlying transition kernel.  Fix $\varrho_* \in (0,K)$.  We start the process from a perturbation of the constant profile
\begin{equation}\label{rhoNini}
\varrho_N^{\rm ini} (u):= \varrho_* + N^{-\alpha} \rho^{\rm ini} (u),
\end{equation}
where $\rho^{\rm ini} \in C^{\infty} (\T^d)$. We speed up the process by $N^{1+\kappa}$, $\kappa \leq \alpha$, and choose the reference measure associated to the \emph{smooth} profile
\[\vr_*+N^{-\alpha}\rho(N^{\kappa-\alpha}t,u-N^\kappa\la tm),\]
where $\lambda := J^\prime (\varrho_*) = K - 2 \varrho_*$ and $\rho$ is the classical solution to the Burgers equation
\begin{equation}\label{burgers}
\partial_s \rho(s,u) - m \cdot \nabla\big(\rho^2(s,u)\big) = 0, \quad \rho(0,u)=\rho^{\rm ini}(u),
\end{equation}
up to the first shock appears. Then, in Theorem \ref{thm1}, we show that under some restrictions on $\kappa$ and $\alpha$, the relative entropy is of order $o(N^{d-2\alpha})$.  In particular, we partially extend the results in \cite{toth2002between} to $0 < \alpha < 1/3$ if $d = 1$, and to $0 < \alpha < 1$ if $d  \geq 2$.  As a result, in Theorem \ref{coro1}, we show the perturbed quantity evolves according to \eqref{burgers}.

The proof relies on the refined relative entropy method recently introduced by Jara and Menezes\;\cite{jara2018non,jaram18nonequilireaction}, also used in many other contexts, e.g., \cite{funaki2019motion,jara2020stochastic,JLT21}. Compared to the technique in \cite{toth2002between}, our proof does not involve the spectral gap estimate or logarithmic Sobolev inequality for the underlying microscopic dynamics, which is known as a hard problem for general interacting particle systems. It also remains as an interesting question to check whether the techniques extend to cover also the case $d \geq 3$ and $\alpha = \kappa =1$ in \cite{esposito1994diffusive}.

\subsubsection{Anharmonic chain of oscillators}

In Section \ref{sec:chain} and \ref{sec:chain-rel-ent}, we study the one-dimensional \emph{chain of (unpinned) anharmonic oscillators} with conservative noise.
The noise is modeled by Langevin thermostats acting at each position and fixing the temperature to $T=\beta^{-1}$.
The only two conserved quantities of the dynamics are the momentum $\fp$ and the length $\fr$, hence it corresponds to the case $n=2$, $d=1$.
The hydrodynamic equation is given by the following $p$-system
\begin{align*}
  \partial_t\fp=\partial_v\bst(\fr), \quad \partial_t\fr=\partial_v\fp, \quad (\fp,\fr)(0,\cdot)=(\fp,\fr)^\ini,
\end{align*}
where $\bst =\bst_\beta (\fr)$ is the equilibrium tension.
It is proved in \cite{EvenO14} together with the energy conservation law in the smooth regime and in \cite{MarcheO18} after the appearance of shocks.

Define $J(\fp,\fr):=(-\bst(\fr),-\fp)$ and fix $(\fp_*,\fr_*)\in\bR^2$ such that $\bst'(\fr_*)\not=0$.
Denote by $A=DJ(\fp_*,\fr_*)$ and note that $A\bsu_-=\sqrt{\bst'(\fr_*)}\bsu_-$, $A\bsu_+=-\sqrt{\bst'(\fr_*)}\bsu_+$ for $\bsu_\pm=(\pm\sqrt{\bst'(\fr_*)},1)'$.
Similarly to the exclusion process, we start the dynamics from a perturbed profile
\begin{equation*}
  (\fp_*, \fr_*)'+N^{-\alpha}\sigma_-^\ini\bsu_-+N^{-\alpha}\sigma_+^\ini\bsu_+,
\end{equation*}
where $\sigma_\pm^\ini \in C^1(\bT)$.
We also speed up the time to be $N^{1+\kappa}t$ for some $\kappa\le\alpha$.
Choose the reference measure associated to the slowly varying parameters
\begin{equation*}
  (\mathfrak p_*, \mathfrak r_*)'+N^{-\alpha}\sum_{j=\pm} \sigma_j\big(N^{\kappa-\alpha}t,v+jN^\kappa\sqrt{\bst'(\fr_*)}t\big)\bsu_j,
\end{equation*}
where $(\sigma_-,\sigma_+)$ is the smooth solution to
\begin{equation}
\label{eq:chain-burgers}
  \begin{aligned}
    &\partial_s\sigma_-(s,u)+\frac{\boldsymbol\tau''(\mathfrak r_*)}{4\sqrt{\boldsymbol\tau'(\mathfrak r_*)}}\partial_u\big(\sigma_-^2(s,u)\big)=0, \quad \sigma_-(0,u)=\sigma_-^\ini(u),\\
    &\partial_s\sigma_+(s,u)-\frac{\boldsymbol\tau''(\mathfrak r_*)}{4\sqrt{\boldsymbol\tau'(\mathfrak r_*)}}\partial_u\big(\sigma_+^2(s,u)\big)=0, \quad \sigma_+(0,u)=\sigma_+^\ini(u).
  \end{aligned}
\end{equation}
We prove in Theorem \ref{thm:chain-rel-ent} that in the smooth regime, the relative entropy grows with the order $o(N^{1-2\alpha})$.
As a consequence, we prove in Theorem \ref{thm:chain-equi-pert} that the macroscopic perturbation is governed by the decoupled system \eqref{eq:chain-burgers}.

The proof of the relative entropy estimate is based on the equivalence of ensembles for inhomogeneous Gibbs states and the uniform gradient estimate for the Poisson equation associated to the generator of the stochastic dynamics, \emph{cf.} \cite[Section 8 \& 9]{Xu20}.
It is worth pointing out that, unlike the exclusion, the equilibrium perturbation for Hamiltonian dynamics has not been investigated before, mainly due to the technical difficulties in obtaining fine estimate for the relative entropy.

\section{Asymmetric generalized exclusion process}
\label{sec:gasep}

In this and the next sections, we consider the asymmetric generalized exclusion process. Fix a positive integer $K > 0$, which denotes the maximum number of particles at each site.    The state space of the generalized exclusion process is $\Omega_N^d:= \{0,1,\ldots,K\}^{\T_N^d}$. For a configuration $\eta \in \Omega_N^d$, $\eta_x$ denotes the number of particles at site $x\in\bT_N^d$.
Denote by $\{e_i\}_{1 \leq i \leq d}$ the canonical basis of $\bT_N^d$ and let $e_i=-e_{i-d}$ for $i=d+1$, ..., $2d$.
Given the jump rates $\{p_i\ge0;1 \le i \le 2d\}$, consider the generator $\cL_N$ which is given for any function $f: \Omega_N^d \rightarrow \R$ by
\begin{align*}
  \cL_N f(\eta) = \sum_{i=1}^{2d} \sum_{x\in\bT_N^d} c_{x,i}(\eta)\big[f(\eta^{x,x+e_i})-f(\eta)\big],
\end{align*}
where the jump rate is given by $c_{x,i}(\eta):=p_i\eta_x(K-\eta_{x+e_i})$
and $\eta^{x,y}$ is the configuration obtained from $\eta$ after a particle jumps from $x$ to $y$,
\[(\eta^{x,y} )_z= \begin{cases}
	\eta_x - 1, \quad z = x,\\	
	\eta_y + 1, \quad z = y,\\
	\eta_z, \quad \text{otherwise.}\\
\end{cases}\]
Assume that $p_i+p_{i+d}>0$ for each $1 \le i \le d$.
Denote $m=(m_i)_{1 \le i \le d}\in\bR^d$ where $m_i=p_i-p_{i+d}$.
Also assume that $m$ is a nonzero vector, hence $\cL_N$ is asymmetric.

The generator $\cL_N$ has a family of  product invariant measures indexed by the particle density.
For $\vr\in[0,K]$, let $\nu_\vr^1$ be the binomial measure $B(K,\vr K^{-1})$:
\begin{align*}
  \nu_\vr^1(k):=\binom Kk \left( \frac\vr K \right)^k \left( 1-\frac\vr K \right)^{K-k}, \quad \forall\,k=0,\ldots,K.
\end{align*}
Denote by $\nu_\vr^N$ the product measure on $\Omega_N^d$ such that $\nu_\vr^N(\eta_x)=\nu_\vr^1(\eta_x)$ for each $x\in\bT_N^d$.
It is not hard to check that the family $\{\nu_\vr^N; 0 \le \vr \le K\}$ is invariant under $\cL_N$.
Observe that the average number of particles per site under $\nu_\vr^N$ is $\vr=E_{\nu_\vr^1} [\eta_x]$.

Let $\mu_{N,0}$ be associated to the profile $\vr_N^\ini$ in \eqref{rhoNini} with some $\alpha>0$, i.e., for any $\vf \in C(\bT^d)$ and any $\ve>0$,
\[\lim_{N \rightarrow \infty} \mu_{N,0} \Big\{ \Big|\frac{1}{N^{d-\alpha}}\sum_{x \in \bT_N^d}( \eta_x - \varrho_*)\vf \big(\tfrac{x}{N}\big) - \int_{\bT^d} \rho^{\rm ini} (u) \vf(u) du\Big| > \ve\Big\}=0.\]
For $0<\kappa\le\alpha$, denote by $\{\eta=\eta(t);t\ge0\}$ the Markov process generated by $N^{1+\kappa}\cL_N$ starting from $\mu_{N,0}$.
Recall that $\la=K-2\varrho_*$ and define
\begin{align}
\label{rhoN}
  \vr_N(t,u):=\vr_*+N^{-\alpha}\rho(N^{\kappa-\alpha}t,u-N^\kappa\la tm),
\end{align}
where $\rho=\rho(t,u)$ is the smooth solution to \eqref{burgers}.
Note that for $N$ sufficiently large, $\vr_N\in(0,K)$.
Denote by $\mu_{N,t}$ the distribution of $\eta(t)$ and by $\nu_{N,t}$ the slowly varying product measure
\begin{align}\label{nuNt}
  \nu_{N,t} (d \eta)=  \bigotimes_{x \in \T_N^d} \nu_{\vr_{N,x}}^1 (d \eta_x), \quad \vr_{N,x}:=\vr_N \left( t,\frac xN\right).
\end{align}
Let $f_{N,t}(\eta):=\nu_{N,t}(\eta)^{-1}\mu_{N,t}(\eta)$ be the Radon--Nikodym derivative.
For a probability measure $\mu$ on $\Omega_N^d$ and some $\mu$--density $f$, \emph{i.e.} $f \geq 0$ and $\int_{\Omega_N^d} f d \mu = 1$, define the relative entropy as
\begin{align*}
  H(f;\mu) := \int_{\Omega_N^d} f \log f  d \mu = \sum_{\eta\in\Omega_N^d} f(\eta)\log f(\eta)\mu(\eta).
\end{align*}
To shorten the notations, denote
\begin{align*}
  H_{N} (t) := H (f_{N,t}; \nu_{N,t}), \quad \forall\,t\in[0,T].
\end{align*}

It is well known that the entropy solution to the Burgers equation \eqref{burgers} may develop shocks in a finite time interval even if the initial density is smooth, \emph{cf.}\;\cite{malek1996} for example.  Since our proof depends on the relative entropy method, which requires the reference density profile to be smooth, throughout the article, we assume the initial density profile $\rho^{\rm ini}$ is smooth, and when $\kappa=\alpha$, fix a time horizon $T > 0$ such that the solution to the Burgers equation \eqref{burgers} is smooth during the time interval $[0,T]$.  

\medspace

The followings are the main results of this part.

\begin{theorem}\label{thm1} 
Suppose that $H_N(0)=o(N^{d-2\alpha})$.
\begin{itemize}
\item[(\romannum1)]
If $d=1$, $\alpha\in(0,1/2)$, then for any $t \in [0,T]$, $H_N(t)=o(N^{1-2\alpha})$ for $\kappa\in(0,\alpha]\cap(0,1-2\alpha)$;
\item[\romannum2)]
If $d\ge2$, $\alpha>0$, then for any $t \in [0,T]$,  $H_N(t)=o(N^{d-2\alpha})$ for $\kappa\in(0,\alpha]\cap(0,1)$.
\end{itemize}
\end{theorem}

As a direct consequence of the above theorem, we have the following law of large numbers for the perturbed quantities.

\begin{theorem}\label{coro1}
Suppose the assumptions in Theorem \ref{thm1} hold and further that $\alpha<d/2$.
Then, for any $\vf \in C(\bT^d)$ and any $\ve>0$,
\[    \lim_{N\to\infty} \mu_{N,t} \bigg\{ \bigg| \frac 1{N^{d-\alpha}} \sum_{x\in\bT_N^d} \big(\eta_x -\vr_*\big) \varphi \left( \frac xN-N^\kappa\la tm \right)\\
-\int_{\bT^d} \rho^{(\alpha,\kappa)}(t,u)\varphi(u)du \bigg| > \ve \bigg\} = 0,\]
with the function $\rho^{(\alpha,\kappa)}$ defined as
\[\rho^{(\alpha,\kappa)} (t,u) = \begin{cases}
	\rho^{\rm ini} (u), \quad \kappa < \alpha,\\
	\rho (t,u), \quad \kappa = \alpha,
\end{cases}\]
where $\rho (t,u)$ is the smooth solution to \eqref{burgers}.
\end{theorem}

\begin{remark}
		It is believed that in dimension $d=1$, the above result extends to $\kappa\in(0,\alpha]$ even beyond the appearance of shocks, cf.\;\cite{toth2002between}.
\end{remark}

\begin{remark}
The results in the above theorem could be interpreted in terms of \emph{second class} particles. For simplicity, consider the case $K=1$ and we refer the readers to \cite{giardina2009duality} for $K >1$. The dynamic is defined as follows: there are two kinds of particles, called first and second class particles, in the system.   On top of the exclusion rule, the first class particles have priorities to jump over the second class ones. Precisely speaking, if a first class particle jumps to a site occupied by a second class one, then the jump is performed and the two particles exchange their positions; while if a  second class particle jumps to a site occupied by a first class one, then the jump is suppressed. At the initial time,  independently at each site $x$, put one first class particle with probability $\varrho_*$, and one second class particle with probability $N^{-\alpha} \rho^{\rm ini} (x/N)$. Then, at the macroscopic time $t$, the density profile of the second class particles, along the characteristic line of the PDE \eqref{hyperPDE}, is described by the function $\rho^{\alpha,\kappa}$.

Indeed, denote by $\eta^1_x (t) $ (resp.\;$\eta^2_x (t)$) the number of the first (resp.\;the second) class particles at site $x$ at time $t$, then $\eta_x (t)= \eta_x^1 (t) + \eta_x^2 (t)$.  Since the process of the first class particle is in equilibrium, the distribution of $\eta^1 (t)$ is given by $\nu^N_{\varrho_\ast}$ for any time $t \geq 0$. Therefore, for any $\varphi \in C(\T^d)$, 
\[{\rm Var} \Big( \frac{1}{N^{d-\alpha}} \sum_{x \in \T_N^d} (\eta_x^1 (t) - \varrho_*) \varphi \big( \tfrac{x - \lambda t N^{1+\kappa} m}{N}\big) \Big) = \mathcal{O}  \big(N^{2\alpha-d}\big),\]
which converges to zero as $N \rightarrow \infty$ since $\alpha< d/2$. As a result, for any $\ve > 0$,
\[\lim_{N\to\infty} \mu_{N,t} \bigg\{ \bigg|  \frac{1}{N^{d-\alpha}} \sum_{x \in \T_N^d} (\eta_x^1 - \varrho_*) \varphi \big( \tfrac{x - \lambda t N^{1+\kappa} m}{N} \big) \bigg| > \ve \bigg\} = 0.\]
Together with Theorem \ref{coro1},   we have
\[\lim_{N\to\infty} \mu_{N,t} \bigg\{ \bigg|     \frac{1}{N^{d-\alpha}} \sum_{x \in \T_N^d} \eta_x^2  \varphi \big( \tfrac{x - \lambda t N^{1+\kappa} m}{N}\big) - \int_{\T^d} \rho^{\alpha,\kappa}(t,u) \varphi (u)\,du       \bigg| > \ve \bigg\} = 0.\]
\end{remark}

\section{Relative entropy for the generalized exclusion}
\label{sec:gasep-rel-ent}

In this section,  we prove Theorems \ref{thm1} and \ref{coro1}.
Recall the profile $\vr_N$ defined in \eqref{rhoN}.
To make notations simple, in the following calculations, we denote
\begin{align*}
  \vr_x = \vr_{N,x}(t) := \vr_N \left (t,\frac xN \right),
\quad \forall\,x\in\bT_N^d.
\end{align*}

For a probability measure $\mu$ on  $\Omega_N^d$ and a $\mu$--density $f$,  define the Dirichlet form
\begin{align}\label{ep:dirichlet}
  D_N(f;\mu):=\sum_{i=1}^{2d} \sum_{x\in\bT_N^d} \sum_{\eta\in\Omega_N^d} c_{x,i}(\eta) \left( \sqrt{f(\eta^{x,x+e_i})}-\sqrt{f(\eta)} \right)^2\mu(\eta).
\end{align}
We claim that there exists $\delta_0 = \delta_0 (\vr_*)$ such that for any $\nu_{N,t}$--density $f$,
\begin{equation}\label{eqn:dirichlet}
  D_N(f;\nu_{N,t}) \ge \delta_0\sum_{i=1}^{2d} \sum_{x\in\bT_N^d} \sum_{\eta\in\Omega_N^d} \eta_x(K-\eta_{x+e_i}) \left( \sqrt{f(\eta^{x,x+e_i})}-\sqrt{f(\eta)} \right)^2 \nu_{N,t}(\eta).
\end{equation}
Indeed, the claim holds obviously if $p_i>0$ for all $1 \leq i \leq 2d$.
Suppose that $p_i= 0$ for some $1 \le i \le d$, then $p_{i+d}\not=0$ due to our assumption.
Recalling that $e_{i+d}=-e_i$,
\begin{multline*}
  \sum_\eta \eta_x(K - \eta_{x-e_i}) \left( \sqrt{f(\eta^{x,x-e_i})}-\sqrt{f(\eta)} \right)^2 \nu_{N,t} (\eta)\\
= \sum_\eta (\eta_x +1)(K - \eta_{x-e_i}+1) \left( \sqrt{f(\eta)}-\sqrt{f(\eta^{x-e_i,x})} \right)^2 \nu_{N,t} (\eta^{x-e_i,x}),
\end{multline*}
for each $x\in\bT_N^d$.
Observe from the definition of $\nu_{N,t}$ in \eqref{nuNt} that
\begin{align*}
  \frac{\nu_{N,t}(\eta^{x-e_i,x})}{\nu_{N,t}(\eta)}=\frac{\theta_{x}}{\theta_{x-e_i}}\frac{\eta_{x-e_i}(K-\eta_x)}{(K-\eta_{x-e_i}+1)(\eta_x+1)},
\end{align*}
where 
\[ \theta_x  := \frac{\vr_x}{K-\vr_x}.\]
 Since $|\vr_x-\vr_*| \le N^{-\alpha}||\rho||_\infty$ and $\vr_*\in(0,K)$, hence there exists some $C=C(\vr_*)>0$, such that $C^{-1}<\theta_x<C$ for all $x$ and sufficiently large $N$.
Therefore,
\begin{multline*}
  \sum_\eta \eta_x(K - \eta_{x-e_i}) \left( \sqrt{f(\eta^{x,x-e_i})}-\sqrt{f(\eta)} \right)^2 \nu_{N,t} (\eta)\\
\ge C'\sum_\eta \eta_{x-e_i}(K - \eta_x) \left( \sqrt{f(\eta)}-\sqrt{f(\eta^{x-e_i,x})} \right)^2 \nu_{N,t}(\eta),
\end{multline*}
which is enough to conclude the claim.

Recall that $f_{N,t} = \mu_{N,t} / \nu_{N,t}$ and let $\psi_{N,t}:=(\nu_{\vr_*}^N)^{-1}\nu_{N,t}$ be the derivative of $\nu_{N,t}$ with respect to the stationary measure $\nu_{\vr_*}^N$.
By Yau's relative entropy inequality (see Lemma \ref{lem:relativeEntIne}),
\begin{align*}
  \frac{d}{dt} H_N (t) \leq &- N^{1+\kappa} D_N (f_{N,t};\nu_{N,t})\\
  &+\sum_{\eta\in\Omega_N^d} \left[ N^{1+\kappa}\Lcal_{N,t}^\ast \mathbf{1} (\eta) - \frac{d}{dt} \log \psi_{N,t} (\eta) \right] \mu_{N,t}(\eta),
\end{align*}
where $\Lcal_{N,t}^\ast$ is the adjoint of $\Lcal_N$ with respect to $\nu_{N,t}$.  The main idea is to write the right-hand side of the above inequality as $CH_N(t) + o(N^{d-2\alpha})$ for some finite constant $C$ independent of $N$, and the result then follows from  Gr{\" o}nwall's inequality.

Since $c_{x,i}(\eta)$ is the jump rate from $\eta\in\Omega_N^d$ to $\eta^{x,x+e_i}$, Lemma \ref{lem:relativeEntIne} yields that
\begin{equation*}
  \begin{aligned}
    \cL_{N,t}^* \mathbf1(\eta) &= \sum_{i=1}^{2d} \sum_{x\in\bT_N^d} \left\{ \frac{\nu_{N,t}(\eta^{x+e_i,x})}{\nu_{N,t}(\eta)}c_{x,i}(\eta^{x+e_i,x})-c_{x,i}(\eta) \right\}\\
    &= \sum_{i=1}^{2d} \sum_{x\in\bT_N^d} p_i \left\{ \frac{\vr_x(K-\vr_{x+e_i})}{\vr_{x+e_i}(K-\vr_x)}\eta_{x+e_i}(K-\eta_x) - \eta_x(K-\eta_{x+e_i}) \right\}\\
    &= \sum_{i=1}^{2d} \sum_{x\in\bT_N^d} p_i\vr_x(K-\vr_{x+e_i}) \left\{ \frac{\eta_{x+e_i}(K-\eta_x)}{\vr_{x+e_i}(K-\vr_x)}-\frac{\eta_x(K-\eta_{x+e_i})}{\vr_x(K-\vr_{x+e_i})} \right\}.
  \end{aligned}
\end{equation*}

For any configuration $\eta \in \Omega_N^d$ and any site $x \in \T_N^d$, denote
\begin{equation}\label{omegaX}
\omega_x = \omega_x(\eta) := \frac{\eta_x - \vr_x}{\vr_x(K-\vr_x)}.
\end{equation}
It could be checked directly that for $x \neq y$,
\begin{equation}\label{re1}
\frac{\eta_x(K-\eta_y)}{\vr_x (K-\vr_y) } -  \frac{\eta_y(K-\eta_x)}{\vr_y (K-\vr_x) }= K \big[\omega_x -  \omega_y + (\vr_x - \vr_y) \omega_x \omega_y\big].
\end{equation}
Recall that $e_{i+d}=-e_i$ for $1 \leq i \leq d$, and \eqref{re1} permits us to write $\Lcal_{N,t}^\ast \mathbf{1} (\eta)$ as
\[	\Lcal_{N,t}^\ast \mathbf{1} (\eta) = \sum_{i=1}^d \sum_{x \in \T_N^d}  b_N (x,x+e_i) \big\{ \big( \omega_{x+e_i} - \omega_x\big) 
+ (\varrho_{x+e_i} - \varrho_x) \omega_x \omega_{x+e_i} \big\},\]
where for $1 \leq i \leq d$,
\[b_N (x,x+e_i) := K \big[p_i \varrho_x(K-\vr_{x+e_i}) - p_{i+d} \vr_{x+e_i} (K-\vr_x)  \big].\]

Now we calculate $\frac{d}{dt} \log \psi_{N,t} (\eta)$. Since $\nu_{N,t}$ and $\nu^N_{\varrho_*}$ are both product measures,
\[\log \psi_{N,t} (\eta) = \sum_{x \in \T_N^d} \left[ \eta_x \log \left( \frac{\theta_x}{\vr_*} \right) - K \log \left( \frac{1+\theta_x}{1+\vr_*} \right) \right].\]
The time derivative then reads
\begin{multline*}
  \frac d{dt}\log \psi_{N,t} (\eta) = \sum_{x \in \T_N^d} \Big\{ \eta_x  \frac{\partial_t \theta_x}{\theta_x} - K \frac{\partial_t \theta_x}{1+\theta_x} \Big\} \\
=  \sum_{x \in \T_N^d} K \Big\{ \eta_x \frac{\partial_t \vr_x}{\vr_x (K-\vr_x)} -  \frac{\partial_t \vr_x}{K-\vr_x} \Big\}= \sum_{x \in \T_N^d} K \omega_x  \partial_t \vr_x.
\end{multline*}

To sum up, we have shown that
\begin{equation}\label{eqn:relatEnt}
\frac{d}{dt} H_N (t) \leq - N^{1+\kappa} D_N (f_{N,t};\nu_{N,t}) + \cR_{N,t} + \Ecal_{N,t},
\end{equation}
where
\begin{eqnarray}
\cR_{N,t}  &=& \sum_{\eta \in \Omega_N^d} \sum_{i=1}^d  \sum_{x \in \T_N^d} a^N_{i,x} \omega_x \omega_{x+e_i} \mu_{N,t}(\eta),\label{RNt}\\
a_{i,x}^N &=& N^{1+\kappa} b_N (x,x+e_i) (\varrho_{x+e_i} - \varrho_{x}),\label{aNix}
\end{eqnarray}
and
\[\Ecal_{N,t} = \sum_{\eta \in \Omega_N^d}  \Big\{  \sum_{i=1}^d  \sum_{x \in \T_N^d}N^{1+\kappa}b_N (x,x+e_i) (\omega_{x+e_i}- \omega_x ) \\
- \sum_{x \in \T_N^d} K  \omega_x \partial_t \varrho_{x} \Big\} \mu_{N,t} (\eta).\]

\subsection{Proofs of Theorems \ref{thm1} and \ref{coro1}} In this subsection, we prove Theorems \ref{thm1} and \ref{coro1}. We first deal with the error term $\Ecal_{N,t}$.

\begin{lemma}
There exists a constant $C$ independent of $N$ such that
	\[\cE_{N,t}\le H_N(t)+CN^{d+2\kappa-2\alpha-2}.\]
\end{lemma}

\begin{proof}
For a sequence $\{a_x;x\in\bT_N^d\}$ and $1 \le i \le d$, let $\nabla_ia_x:=a_{x+e_i}-a_x$ and $\nabla_i^*a_x:=a_{x-e_i}-a_x$.
Using the summation by parts formula,
\begin{align*}
  \cE_{N,t} = \sum_{\eta \in \Omega_N^d} \mu_{N,t}(\eta) \sum_{x \in \T_N^d} \omega_x \left[ N^{1+\kappa}\sum_{i=1}^d \nabla_i^*\big(b_N(x,x+e_i)\big)-K\partial_t\vr_x \right].
\end{align*}
By the definition of $b_N$ and Taylor's expansion,
\begin{multline*}
  \nabla_i^*\big(b_N(x,x+e_i)\big) = N^{-1-2\alpha}Km_i\partial_{u_i}\rho^2 \left( N^{\kappa-\alpha}t,\frac xN-N^\kappa\la tm \right)\\
  -N^{-1-\alpha}K\la m_i\partial_{u_i}\rho \left( N^{\kappa-\alpha}t,\frac xN-N^\kappa\la tm \right)+\ep_{N,i}(x),
\end{multline*}
where $\ep_{N,i}(x)=\mathcal O(N^{-2-\alpha})$.
Meanwhile, by the definition of $\vr_N$ in \eqref{rhoN},
\begin{align*}
  \partial_t\vr_x = \big[N^{\kappa-2\alpha}\partial_s-N^{\kappa-\alpha}\la (m\cdot\nabla_u)\big]\rho \left( N^{\kappa-\alpha}t,\frac xN-N^\kappa\la tm \right).
\end{align*}
Since $\rho$ solves the Burgers equation \eqref{burgers}, we can rewrite the $\cE_{N,t}$ as
\begin{align*}
  \cE_{N,t} = N^{1+\kappa} \sum_{\eta \in \Omega_N^d} \sum_{i=1}^d \sum_{x \in \T_N^d}  \omega_x \epsilon_{N,i} (x) \,\mu_{N,t} (\eta).
\end{align*}
It is easy to see $|\Ecal_{N,t}| \leq C N^{d-1+\kappa-\alpha}$. To get a better bound, by entropy inequality,
\begin{multline}\label{eqn:error1}
\Ecal_{N,t} \leq H_N(t) + \log \bigg[ \sum_{\eta \in \Omega_N^d} \exp \Big\{N^{1+\kappa}\sum_{i=1}^d \sum_{x \in \T_N^d}  \omega_x \epsilon_{N,i} (x) \Big\}\,\nu_{N,t} (\eta) \bigg]\\
= H_N(t) + \sum_{x \in \T_N^d} \log \bigg[ \sum_{\eta \in \Omega_N^d} \exp \big\{  N^{1+\kappa}\omega_x \sum_{i=1}^d \epsilon_{N,i} (x) \big\}\, \nu_{N,t} (\eta) \bigg].
\end{multline}
Note that $\omega_x$ is bounded and  has zero mean with respect to the measure $\nu_{N,t}$. Using the basic inequality 
\[e^a \leq 1 + a + (1/2)a^2 e^{|a|}, \quad \log (1+a) \leq a,\]
we have 
\begin{equation}\label{eqn:error1bound}
\Ecal_{N,t} \leq H_N(t) + \mathcal{O} (N^{d-2+2\kappa-2\alpha}).
\end{equation}
This concludes the proof of the lemma.
\end{proof}

The following result bounds the term $\cR_{N,t}$, whose proof is postponed to Subsection \ref{subsec:bg}.

\begin{proposition}\label{prop:bg}
For $\ell \geq 1$, let 
\begin{equation}\label{gd}
	g_d (\ell) = \begin{cases}
		\ell, &\quad d=1,\\	
		\log \ell, &\quad d=2,\\
		1, &\quad d \geq 3,
	\end{cases}
\end{equation}
as in \eqref{gdl}. Then, there exists a constant $C$ independent of $N$, $\ell$ such that
\begin{multline*}
  \cR_{N,t} \le N^{1+\kappa}D_N (f_{N,t};\nu_{N,t})+C \left[ N^{\kappa-\alpha} + N^{\kappa-2\alpha-1}\ell^dg_d(\ell) \right] \left[ H_N(t) + \frac{N^d}{\ell^d} \right]\\
+C \big[H_N(t) + N^{d+2\kappa-4\alpha-2}\ell^{d}g_d (\ell)\big].
\end{multline*}
\end{proposition}

Now, we are ready to prove Theorem \ref{thm1}.

\begin{proof}[Proof of Theorem \ref{thm1}]
In Proposition \ref{prop:bg}, take 
 \[\ell = \ell (N) = \begin{cases}
	N^{(1+2\alpha-\kappa)/2}, &\quad d=1,\\
	N^{(1+2 \alpha -\kappa)/2} / \sqrt{\log N}, &\quad d = 2,\\
	N^{(1+2\alpha-\kappa)/d}, &\quad d \geq 3.
\end{cases}\]
Therefore, in dimension $d = 1$, if  $ \kappa < 1-2\alpha$ and $\kappa \leq \alpha$, then
\[\frac{d}{dt} H_N(t) \leq C(1+N^{\kappa-\alpha}) H_N (t) + C \Big( N^{(1-2\alpha+\kappa)/2} + N^{\kappa-2\alpha} + N^{2\kappa-2\alpha-1}\Big) \leq C H_N(t) + o(N^{1-2\alpha}).\]
In dimension $d=2$,   if $ \kappa < 1$ and $\kappa \leq \alpha$, then
\[\frac{d}{dt} H_N(t) \leq C(1+N^{\kappa-\alpha}) H_N (t) + C \Big( N^{1-2 \alpha +\kappa} \log N + N^{2\kappa-2\alpha}\Big) \leq  C H_N(t) + o(N^{2-2\alpha}).\]
In dimension $d \geq  3$, if $ \kappa < 1$ and $\kappa \leq \alpha$,  then
\[\frac{d}{dt} H_N(t) \leq C(1+N^{\kappa-\alpha}) H_N (t) + C \Big( N^{d-2\alpha+\kappa-1} + N^{d+2\kappa-2\alpha-2}\Big) \leq  C H_N(t) + o(N^{d-2\alpha}).\]
We conclude the proof by using Gr{\" o}nwall's inequality.
\end{proof}

In view of Theorem \ref{thm1}, the proof of Theorem \ref{coro1} is a direct application of the entropy inequality, \emph{cf.}\,\cite[Corollary 6.1.3]{klscaling} for example. We sketch it below for completeness.

\begin{proof}[Proof of Theorem \ref{coro1}]
It suffices to prove
\[\lim_{N \rightarrow \infty}  \sum_{\eta \in \Omega_N^d} \Big| \frac{1}{N^{d-\alpha}} \sum_{x \in \T_N^d} (\eta_x  - E_{\nu_{N,t}} [\eta_x]) \varphi \big( \tfrac{x - \lambda t N^{1+\kappa} m}{N}\big) \Big|  \mu_{N,t} (\eta) = 0.\]
By the entropy inequality, for any $\gamma > 0$, the integral above is bounded by 
\[\frac{H_N(t)}{\gamma N^{d-2 \alpha}} + \frac{1}{\gamma N^{d-2\alpha}} \log \bigg[ \sum_{\eta \in \Omega_N^d} \exp \Big\{  \Big| \gamma N^{-\alpha}  \sum_{x \in \T_N^d} (\eta_x - E_{\nu_{N,t}} [\eta_x]) \varphi \big( \tfrac{x - \lambda t N^{1+\kappa} m}{N}\big) \Big|  \Big\}  \nu_{N,t} (\eta) \bigg]. \]
By Theorem \ref{thm1}, the first term above converges to zero as $N \rightarrow \infty$. Since 
\[e^{|c|} \leq e^c + e^{-c}, \quad \log(a+b) \leq \log2 +  \max \{\log a, \log b\}\]
for any $c$ and any positive reals $a,b$, and since $\alpha < d/2$, we could remove the absolute value inside the exponential in the second term above, and rewrite it as 
\[\frac{1}{\gamma N^{d-2\alpha}} \sum_{x \in \T_N^d} \log \bigg[ \sum_{\eta \in \Omega_N^d} \exp \Big\{  \gamma N^{-\alpha}  (\eta_x - E_{\nu_{N,t}} [\eta_x]) \varphi \big( \tfrac{x - \lambda t N^{1+\kappa} m}{N}\big)  \Big\}  \nu_{N,t} (\eta) \bigg].\]
Using the basic inequality $e^a \leq 1 + a + (1/2)a^2 e^{|a|}$ and $\log (1+a) \leq a$, there exists a finite constant $C$ independent of $N$ such that the above formula is bounded by
\[\frac{C}{\gamma N^{d-2\alpha}} \times N^d \times (\gamma N^{-\alpha})^2 = C \gamma.\]
Since $\gamma$ could be taken arbitrarily small, the proof is completed.
\end{proof}

\subsection{Proof of Proposition \ref{prop:bg}.}\label{subsec:bg}

It remains to prove Proposition \ref{prop:bg}. The first step is to properly decompose the term $\cR_{N,t}$ defined in Eq.\;\eqref{RNt}.  
For $\ell \geq 1$, let $\pfrak_\ell (\cdot)$ be the uniform measure on $\Lambda_\ell^d = \{0,1,\ldots,\ell-1\}^d$, \emph{i.e.}\;$\pfrak_\ell (x) = \ell^{-d}$ if $x \in \Lambda_\ell^d$ and $=0$ otherwise.  Let $\mathfrak{q}_\ell = \pfrak_\ell \ast \pfrak_\ell$ be the convolution of $\pfrak_\ell$ with itself, 
\[\qfrak_\ell (z) = \sum_{y \in \T_N^d} \pfrak_\ell (y) \pfrak_\ell (z-y), \quad z \in \T_N^d.\] 
For $x \in \T_N^d$ and $\ell \geq 1$, the spatial average of  $\omega_x$ in the box $\Lambda_{2\ell - 1}^d$ is defined as
\[\omega^\ell_x= \sum_{z\in \T_N^d} \omega_{x+z}  \qfrak_\ell (z).\] 
Define 
\begin{equation}\label{eqn:aNtl}
\cR_{N,t}^\ell  = \sum_{\eta \in \Omega_N^d}  \sum_{i=1}^d \sum_{x \in \T_N^d} a^N_{i,x} \omega_x \omega^\ell_{x+e_i}  \mu_{N,t} (\eta).
\end{equation}
Using the definition of flows introduced in Subsection \ref{subsec:flow} and summation by parts formula,
\begin{multline*}
\cR_{N,t} - \cR_{N,t}^\ell = \sum_{\eta \in \Omega_N^d}  \sum_{i=1}^d \sum_{x,z \in \T_N^d}  a^N_{i,x} \omega_x \omega_{x+e_i+z} (\delta_0 (z) - \qfrak_\ell (z)) \mu_{N,t} (\eta)\\
= \sum_{\eta \in \Omega_N^d} \sum_{i,j=1}^d \; \sum_{x,z \in \T_N^d} a^N_{i,x} \omega_x \omega_{x+e_i+z} (\phi_\ell (z,e_j) - \phi_\ell (z-e_j,e_j))  \mu_{N,t} (\eta)\\
= \sum_{\eta \in \Omega_N^d}    \sum_{i,j=1}^d \; \sum_{x,z \in \T_N^d} a^N_{i,x} \omega_x (\omega_{x+e_i+z} - \omega_{x+e_i+e_j+z} ) \phi_\ell (z,e_j)  \mu_{N,t} (\eta).
\end{multline*}
Make the change of variables $x \mapsto x-z-e_i$, and put
\begin{equation}\label{eqn:hjxl}
h_{j,x}^\ell = \sum_{i=1}^d \sum_{z \in \T_N^d} a^N_{i,x-z-e_i} \omega_{x-z-e_i} \phi_\ell (z,e_j),
\end{equation}
then we have
\[\cR_{N,t} - \cR_{N,t}^\ell = \sum_{\eta \in \Omega_N^d}  \sum_{j=1}^d \sum_{x \in \T_N^d} h_{j,x}^\ell (\omega_x - \omega_{x+e_j}) \,\mu_{N,t} (\eta).\]

To further decompose the term on the right-hand side of the last line, we introduce the following integration by parts formula. 

\begin{lemma}[Integration by parts formula]\label{lem:integParts}
Let $h:\Omega_N^d \rightarrow \R$ be such that $h$ does not depend on the values of $\eta_x $ or $\eta_z$ for some $x,\,z \in \T_N^d$. 
Then, for any  $\nu_{N,t}$--density $f$,
\begin{multline*}
\sum_{\eta \in \Omega_N^d}  h(\eta) \big[\omega_z - \omega_x\big] f (\eta) \nu_{N,t} (\eta)= \frac{1}{K} \sum_{\eta \in \Omega_N^d}  h(\eta) s^N_{x,z} (\eta)  \big( f (\eta^{x,z}) - f (\eta)\big)\, \nu_{N,t} (\eta)\\
- (\varrho_{z} - \varrho_{x}) \sum_{\eta \in \Omega_N^d}  h(\eta)  f(\eta) \omega_z \omega_x \nu_{N,t} (\eta),
\end{multline*}
where
\[s^N_{x,z} (\eta) = \frac{\eta_x(K-\eta_z)}{\varrho_{x}(K-\varrho_{z})}.\]
\end{lemma}

\begin{proof}
For $K = 1$, the result is proved in \cite[Lemma E.1]{jara2018non}, and we extend it to general $K$.  
Since $h$ does not depend on the values of $\eta_x$ and $\eta_z$, using the change of variables $\eta \mapsto \eta^{z,x}$, we rewrite the first term on the right-hand side of the above equation as
\[ \frac{1}{K} \sum_{\eta \in \Omega_N^d}  h (\eta) \Big( \frac{s^N_{x,z} (\eta^{z,x}) \nu_{N,t} (\eta^{z,x})}{\nu_{N,t} (\eta) } - s^N_{x,z} (\eta)  \Big) f(\eta) \nu_{N,t} (\eta). \]
Direct calculations show that
\[\frac{s^N_{x,z} (\eta^{z,x}) \nu_{N,t} (\eta^{z,x})}{\nu_{N,t} (\eta) }  = \frac{\eta_z (K - \eta_x)}{\varrho_{z}  (K - \varrho_{x})} = s^N_{z,x} (\eta).\]
We conclude the proof by Eq.\;\eqref{re1}.
\end{proof}

Since $\phi_\ell$ is supported in $\Lambda_{2\ell - 1}^d$, the value of $h^\ell_{j,x}$ does not depend on those of $\eta_x$ and $\eta_{x+e_j}$ for $1 \leq j \leq d$. Recall $ \mu_{N,t} (\eta) = f_{N,t} (\eta) \nu_{N,t} (\eta)$. By Lemma \ref{lem:integParts}, we have
\begin{multline}\label{eqn:1}
	\cR_{N,t} - \cR_{N,t}^\ell = \frac{1}{K}  \sum_{j=1}^d \sum_{x \in \T_N^d} \sum_{\eta \in \Omega_N^d}  h_{j,x}^\ell s^N_{x+e_j,x} \big( f_{N,t} (\eta^{x+e_j,x}) - f_{N,t} (\eta)\big)\, \nu_{N,t} (\eta)\\
	- \sum_{j=1}^d \sum_{x \in \T_N^d} (\varrho_{x} - \varrho_{x+e_j}) \sum_{\eta \in \Omega_N^d}  h_{j,x}^\ell \omega_x\omega_{x+e_j} f_{N,t} \nu_{N,t} (\eta).
\end{multline}
By Young's inequality, for any $\gamma > 0$, the first term on the right-hand side of the last equation is bounded by
\begin{multline}\label{eqn:2}
\frac{\gamma}{2}	\sum_{j=1}^d \sum_{x \in \T_N^d} \sum_{\eta \in \Omega_N^d}  s^N_{x+e_j,x} \Big[\sqrt{f_{N,t}(\eta^{x+e_j,x})} - \sqrt{f_{N,t}(\eta)}\Big]^2  \nu_{N,t} (\eta) \\
+ \frac{1}{2K^2\gamma}	\sum_{j=1}^d \sum_{x \in \T_N^d} \sum_{\eta \in \Omega_N^d}  (h_{j,x}^\ell)^2 s^N_{x+e_j,x} \Big[ \sqrt{f_{N,t}(\eta^{x+e_j,x})} + \sqrt{f_{N,t}(\eta)}\Big]^2 \nu_{N,t} (\eta). 
\end{multline}
Since $\varrho_* \in (0,K)$, for $N$ large enough, there exists a constant $C = C (\varrho_\ast,K) > 0$ such that $s^N_{x,z} \leq C$ for any $x,z \in \T_N^d$. Then, by \eqref{eqn:dirichlet}, the first term in \eqref{eqn:2} is bounded by $C_1 \gamma  D_N (f_{N,t};\nu_{N,t})$ for some $C_1 = C_1 (\varrho_\ast,K,\delta_0)$. Since for any $\eta \in \Omega_N^d$ and any $x,z \in \T_N^d$,
\[\frac{\nu_{N,t} (\eta^{x,z})}{\nu_{N,t} (\eta)} = \frac{\eta_x(K-\eta_z)}{(K+1-\eta_z)(\eta_z+1)} \frac{\varrho_{z} (K-\varrho_{x})}{\varrho_{x} (K - \varrho_{z})} \leq C := C(\varrho_*,K),\]
by Cauchy-Schwarz inequality and exchange of variables $\eta \mapsto \eta^{x,x+e_j}$,  the second term in \eqref{eqn:2} is bounded from above by
\[\frac{C}{\gamma} \sum_{j=1}^d \sum_{x \in \T_N^d} \sum_{\eta \in \Omega_N^d}  (h_{j,x}^\ell)^2 f_{N,t} (\eta) \nu_{N,t} (\eta).\]
Take $\gamma = N^{1+\kappa} / C_1$, together with \eqref{eqn:1}, then we have shown that 
\begin{multline}\label{decomposition}
\cR_{N,t} \leq N^{1+\kappa} D_N(f_{N,t};\nu_{N,t}) + \cR_{N,t}^\ell + \frac{C}{N^{1+\kappa}} \sum_{j=1}^d \sum_{x \in \T_N^d} \sum_{\eta \in \Omega_N^d}  (h_{j,x}^\ell)^2 f_{N,t} (\eta)\,  \nu_{N,t} (\eta) \\
+ \sum_{j=1}^d \sum_{x \in \T_N^d} (\varrho_{x+e_j} - \varrho_{x}) \sum_{\eta \in \Omega_N^d}  h_{j,x}^\ell \omega_x\omega_{x+e_j} f_{N,t} (\eta) \nu_{N,t} (\eta)
\end{multline} 
for some $C = C (\varrho_\ast,K,\delta_0)$.

\medspace

Next, we shall deal with the last three terms on the right-hand side of the above decomposition \eqref{decomposition} respectively.   We first deal with the term $\cR_{N,t}^\ell$ defined in  \eqref{eqn:aNtl}. Observe that
\begin{multline*}
	\sum_{i=1}^d \sum_{x \in \T_N^d} a^N_{i,x} \omega_x \omega^\ell_{x+e_i} = \sum_{i=1}^d \sum_{x,y,z \in \T_N^d}   a^N_{i,x} \omega_x \omega_{x+e_i+z} \pfrak_\ell (y) \pfrak_\ell (z-y) \\
= \sum_{i=1}^d \sum_{x \in \T_N^d}  \Big(\sum_{y\in \T_N^d} a^N_{i,x-y} \omega_{x-y} \pfrak_\ell (y)\Big) \Big(\sum_{z \in \T_N^d} \omega_{x+e_i+z}  \pfrak_\ell (z)\Big).
\end{multline*}
In the last identity, we reindex $x$ by $x-y$ and $z$ by $z+y$. By entropy inequality, for any $\gamma > 0$, 
\begin{multline}\label{eqn:3}
\cR_{N,t}^\ell \leq \frac{1}{\gamma} \bigg( H_N(t) \\
+ \log \Big[ \sum_{\eta \in \Omega_N^d}  \exp \Big\{ \gamma \sum_{i=1}^d \sum_{x \in \T_N^d}  \Big(\sum_{y\in \T_N^d} a^N_{i,x-y} \omega_{x-y} \pfrak_\ell (y)\Big) \Big(\sum_{z \in \T_N^d} \omega_{x+e_i+z}  \pfrak_\ell (z)\Big) \Big\}  \nu_{N,t} (\eta) \Big] \bigg).
\end{multline}
Note that for $\ell$ large enough, the two random variables 
\begin{align*}
 \Big(\sum_{y\in \T_N^d} a^N_{i,x^\prime-y} \omega_{x^\prime-y} \pfrak_\ell (y)\Big) \Big(\sum_{z \in \T_N^d} \omega_{x^\prime+e_i+z}  \pfrak_\ell (z)\Big) ,\\
  \Big(\sum_{y\in \T_N^d} a^N_{i,x^{\prime\prime}-y} \omega_{x^{\prime\prime}-y} \pfrak_\ell (y)\Big) \Big(\sum_{z \in \T_N^d} \omega_{x^{\prime\prime}+e_i+z}  \pfrak_\ell (z)\Big) 
\end{align*}
are independent with respect to $\nu_{N,t}$ if $||x^\prime-x^{\prime \prime}||_\infty > 3\ell$. Denote 
\[x = x' \quad (\text{mod} \; \ell)  \]
if $x_i = x'_i + \ell k_i$ for some $k_i \in \Z$ and for all $1 \leq i \leq d$. Note that for any $x \in \T_N^d$, there exists a point $x' \in \Lambda_{3\ell - 1}^d$ such that $x = x'$ (mod $3\ell$), and for fixed $x$, if $x^\prime, x^{\prime\prime} = x$ (mod $3\ell$) and $x^\prime \neq x^{\prime\prime}$,  then $||x^\prime - x^{\prime\prime}||_\infty \geq 3 \ell$. By H{\" o}lder's inequality and independence, the second term in \eqref{eqn:3} is bounded by
\begin{multline}\label{eqn:5}
	\frac{1}{d (3 \ell)^d} \sum_{i=1}^d \sum_{x' \in \Lambda_{3\ell - 1}^d}  \\
\log \Big[ \sum_{\eta \in \Omega_N^d} \exp \Big\{  \gamma d (3 \ell)^d \sum_{x: x = x' \;(\text{mod}\; 3 \ell)} \Big(\sum_{y \in \T_N^d} a^N_{i,x-y} \omega_{x-y} \pfrak_\ell (y)\Big) \Big(\sum_{z \in \T_N^d} \omega_{x+e_i+z}  \pfrak_\ell (z)\Big) \Big\}  \nu_{N,t} (\eta) \Big]\\
\leq \frac{1}{d (3 \ell)^d} \sum_{i=1}^d \sum_{x \in \T_N^d}\\
 \log \Big[ \sum_{\eta \in \Omega_N^d} \exp \Big\{  \gamma d (3 \ell)^d  \Big(\sum_{y \in \T_N^d} a^N_{i,x-y} \omega_{x-y} \pfrak_\ell (y)\Big) \Big(\sum_{z \in \T_N^d} \omega_{x+e_i+z}  \pfrak_\ell (z)\Big) \Big\} \, \nu_{N,t} (\eta) \Big].
\end{multline}
We claim that for each $x \in \T_N^d$, the random variable
\[\sum_{y \in \T_N^d} a^N_{i,x-y} \omega_{x-y} \pfrak_\ell (y)\]
is sub-Gaussian of order $C_2 N^{2\kappa-2\alpha} \ell^{-d}$ with respect to the measure $\nu_{N,t}$ for some constant $C_2 = C_2 (\varrho_*,K,\partial_{u_i} \rho)$.  Indeed, first note that  by the definition of $a_{i,x}^N$ in \eqref{aNix}, there exists a constant $C=C(\varrho_*,K)$ such that
\[| a_{i,x}^N |  \leq C ||\partial_{u_i} \rho||_\infty N^{\kappa - \alpha}.\]
Since $\omega_x  \leq C(\varrho_\ast,K)$ for $N$ large enough, and has mean zero with respect to $\nu_{N,t}$, by  Lemma \ref{lem:hoeffding}, the random variable $a^N_{i,x} \omega_x$ is sub-Gaussian of order $C N^{2\kappa-2\alpha}$ for some constant $C=C(\varrho_*,K,\partial_{u_i} \rho)$ with respect to $\nu_{N,t}$. Therefore, for any $\theta \in \R$,
\begin{multline*}
\log \Big[ \sum_{\eta \in \Omega_N^d} \exp \Big\{ \theta \sum_{y \in \T_N^d} a^N_{i,x-y} \omega_{x-y} \pfrak_\ell (y) \Big\}  \nu_{N,t} (\eta) \Big]\\=\sum_{y \in \T_N^d} \log \Big[ \sum_{\eta \in \Omega_N^d} \exp \Big\{ \theta  a^N_{i,x-y} \omega_{x-y} \pfrak_\ell (y) \Big\}  \nu_{N,t} (\eta) \Big]\\
\leq \frac{1}{2} \theta^2 \sum_{y \in \T_N^d} C^2 N^{2\kappa-2\alpha} \pfrak_\ell(y)^2 = \frac{C_2}{2} \theta^2 N^{2\kappa-2\alpha} \ell^{-d}.
\end{multline*}
Similarly, one could prove that   $\sum_{z \in \T_N^d} \omega_{x+e_i+z}  \pfrak_\ell (z)$ is sub-Gaussian of order $C_2 \ell^{-d}$. By Corollary \ref{cor:subGau}, taking $\gamma = (4C_2 d 3^d)^{-1} N^{\alpha-\kappa}$, we bound the  term in \eqref{eqn:5}  by $N^d \log 3 / (3 \ell)^d$. Therefore, there exists a constant $C$ independent of $N$ such that
\begin{equation}\label{estimate:aL}
	\cR_{N,t}^\ell \leq \frac{C}{N^{\alpha-\kappa}} \Big(H_N(t) + \frac{N^d}{\ell^d} \Big).
\end{equation}

\medspace

The third term in \eqref{decomposition} is treated in the same way as above.  Using Lemma \ref{lem:flow}, there exists a constant $C=C (\rho_*,K)$ such that the random variable $h_{j,x}^\ell$ defined in \eqref{eqn:hjxl} is sub-Gaussian of order
\[C N^{2\kappa-2\alpha} \sum_{z \in \T_N^d} \phi_\ell(z,e_j)^2 \leq C_3 N^{2\kappa-2\alpha} g_d (\ell)\]
with respect to the measure $\nu_{N,t}$ for some constant $C_3=C_3 (\varrho_*,K,C_0)$ with $C_0$ introduced in Lemma \ref{lem:flow}. Also note that the two random variables $h^\ell_{j,x^\prime}$ and $h^\ell_{j,x^{\prime\prime}}$ are independent under the measure $\nu_{N,t}$ if $||x^\prime - x^{\prime \prime}||_\infty > 3 \ell$. By entropy inequality and H{\"o}lder's inequality, for any $\gamma > 0$,
\begin{multline}\label{third}
	\frac{1}{N^{1+\kappa}} \sum_{\eta \in \Omega_N^d} \sum_{j=1}^d \sum_{x \in \T_N^d} (h_{j,x}^\ell)^2 f_{N,t}(\eta)\,  \nu_{N,t} (\eta) \\
	 \leq \frac{1}{\gamma N^{1+\kappa}} \bigg( H_N(t) + \log \Big[ \sum_{\eta \in \Omega_N^d} \exp \Big\{ \gamma \sum_{j=1}^d \sum_{x \in \T_N^d} (h_{j,x}^\ell)^2  \Big\}  \nu_{N,t} (\eta)\Big] \bigg)\\
	\leq \frac{1}{\gamma N^{1+\kappa}} \bigg( H_N(t) + \frac{1}{d (3 \ell)^d} \sum_{j=1}^d\sum_{x \in \T_N^d} \log \Big[ \sum_{\eta\in\Omega_N^d} \exp \Big\{  \gamma d(3\ell)^d  (h_{j,x}^\ell)^2  \Big\} \nu_{N,t} (\eta)\Big] \bigg).
\end{multline}
Take $\gamma = \big[4C_3d3^d N^{2\kappa-2\alpha}\ell^d g_d(\ell)\big]^{-1}$, then the assumption of Lemma \ref{lem:subGau} is satisfied, thus there exists a finite constant $C$ independent of $N$ such that the third term in \eqref{decomposition} is bounded by
\begin{equation}\label{estim:third}
 \frac{C \ell^d g_d(\ell)}{N^{1+2\alpha-\kappa}} \Big( H_N(t) + \frac{N^d}{\ell^d}\Big).
\end{equation}

\medspace

Now we deal with the last term in \eqref{decomposition}. First note that
\[|\varrho_{x+e_j} - \varrho_{x}| \leq N^{-1-\alpha} ||\partial_{u_j} \rho||_\infty.\]
By entropy inequality and H{\"o}lder's inequality, for any $\gamma > 0$, we may bound the last term in \eqref{decomposition} by
\begin{multline}\label{eqnlast1}
\frac{1}{\gamma N^{1+\alpha}} \bigg( H_N(t) + \log  \Big[ \sum_{\eta \in \Omega_N^d} \exp \Big\{ \gamma \sum_{j=1}^d \sum_{x \in \T_N^d} N^{1+\alpha} (\varrho_{x+e_j} - \varrho_{x})  h_{j,x}^\ell \omega_x\omega_{x+e_j}  \Big\}  \nu_{N,t} (\eta) \Big] \bigg)\\
\leq \frac{1}{\gamma N^{1+\alpha}} \bigg( H_N(t) \\
+ \frac{1}{d (3 \ell)^d} \sum_{j=1}^d \sum_{x \in \T_N^d}  \log \Big[ \sum_{\eta \in \Omega_N^d} \exp \Big\{ \gamma d(3 \ell)^d  N^{1+\alpha} (\varrho_{x+e_j} - \varrho_{x})  h_{j,x}^\ell \omega_x\omega_{x+e_j}  \Big\}  \nu_{N,t} (\eta)\Big]\bigg).
\end{multline}
Since  $|\omega_x \omega_{x+e_j}| \leq C:=C(\varrho_\ast,K)$, and we have already shown that $h_{j,x}^\ell$ is sub-Gaussian of order $C_3 N^{2\kappa-2\alpha} g_d(\ell)$, then, for any $\theta \in \R$,
\begin{multline}\label{eqn:4}
\log \Big[ \sum_{\eta \in \Omega_N^d} \exp \{ \theta h_{k,x}^\ell \omega_x \omega_{x+e_j}\}  \nu_{N,t} (\eta) \Big] \leq \log \Big[ \sum_{\eta \in \Omega_N^d} \exp \{ C | \theta h_{k,x}^\ell| \}  \nu_{N,t} (\eta) \Big]\\
\leq \log \Big[ \sum_{\eta \in \Omega_N^d} \Big( \exp \{ C \theta h_{j,x}^\ell \} + \exp \{ -C\theta h_{j,x}^\ell \}\Big)  \nu_{N,t} (\eta) \Big] \leq \log 2 + C \theta^2 N^{2\kappa-2\alpha} g_d(\ell)
\end{multline}
for some constant $C=C(\varrho_*,K,C_3)$. Therefore, there exists some constant $C$ independent of $N$, such that  the term on the right-hand side of \eqref{eqnlast1} is bounded by
\[ \frac{C}{\gamma N^{1+\alpha}} \Big[H_N(t) + \frac{N^d}{\ell^d} \times \ell^{2d} g_d (\ell) \gamma^2 N^{2\kappa-2\alpha}+ \frac{N^d}{\ell^d}\Big].\]
Take $\gamma = N^{-1-\alpha}$, then the last line is bounded by
\begin{equation}\label{esti:last}
C \Big(H_N(t) + \ell^d g_d(\ell)N^{d+2\kappa-4\alpha-2}+ \frac{N^d}{\ell^d}\Big).
\end{equation}

\medspace

We conclude the proof of Proposition \ref{prop:bg} by using \eqref{decomposition}, \eqref{estimate:aL}, \eqref{estim:third} and \eqref{esti:last}.

\section{$1$-d chain of anharmonic oscillators}
\label{sec:chain}

In this and the next sections, we consider a chain of $N$ coupled oscillators in one-dimensional lattice space. The configuration space is  $\Omega_N = (\bR^2)^{\bT_N}$, with its elements denoted by $\eta = \{\eta_x=(p_x,r_x); x \in \bT_N\}$. Above, $\bT_N = \bT_N^1$.
All particles have identical mass $1$.
The momentum and position of the particle $x=1$, ..., $N$ are denoted by $p_x\in\bR$ and $q_x\in\bR$, respectively.
The interaction between two particles $x-1$ and $x$ is determined by an anharmonic spring with the potential energy $V(q_x-q_{x-1})$, where $V$ is some nice function.
The total energy is given by the Hamiltonian
\begin{align*}
  \mathcal H_N(\eta):=\sum_{x\in\bT_N} \frac{p_x^2}2 + V (q_x-q_{x-1}), \quad \forall\,\eta\in\Omega_N.
\end{align*}
The corresponding Hamiltonian dynamics then reads
\begin{equation*}
  \dot p_x=-\partial_{q_x}\mathcal H_N, \quad \dot q_x=\partial_{p_x} \mathcal H_N, \quad \forall\,x\in\bT_N.
\end{equation*}

Assume that $V \in C^2(\bR;\bR)$ and some constant $c>0$, such that
\begin{align*}
  c^{-1} \le V''(r) \le c, \quad \forall\,r\in\bR.
\end{align*}
Define $r_x:=q_x-q_{x-1}$ to be the inter-particle distance and require the periodic boundary condition: $(p_{N+1},r_{N+1})=(p_1,r_1)$.

Observe that the total momentum, the total volume, and the Hamiltonian are conserved.
Under a generic assumption of \emph{local equilibrium}, Euler equations can be formally obtained as the evolution of  these quantities.
However, to prove it for the purely deterministic system turns out to be a difficult task.
Proper stochastic noise helps us solve this problem.
Suppose that at each bond $(x, x + 1)$, the deterministic system is in contact with a thermal bath at fixed temperature.
More precisely, fix some inverse temperature $\beta > 0$ and define
\begin{align*}
  \cY_x := \frac\partial{\partial r_{x+1}} - \frac\partial{\partial r_x}, \quad \cY_x^* = \beta\big(V'(r_{x+1}) - V'(r_x)\big) - \cY_x.
\end{align*}
For some deterministic parameter $\gamma_N>0$ that regulates the strength of the heat bath, consider the operator $\cL_N$ given by
\begin{align}
\label{eq:chain-gen}
  \cL_N := \cA_N + \gamma_N\cS_N, \quad \cS_N := -\frac12\sum_{x\in\bT_N} \cY_x^*\cY_x,
\end{align}
where $\cA_N$ is the Liouville operator given by
\begin{align*}
  \cA_{N} := \sum_{x\in\bT_N} (p_x - p_{x-1})\frac\partial{\partial r_x} + \big(V^\prime(r_{x+1}) - V^\prime(r_x)\big)\frac\partial{\partial p_x}.
\end{align*}
The Markov process generated by $\cL_N$ is equivalently expressed by the following system of stochastic differential equations: for each $x\in\bT_N$,
\begin{equation*}
  \left\{
  \begin{aligned}
    dp_x(t) =\ &\big(V'(r_{x+1}) - V' (r_x)\big)dt, \\
    dr_x(t) =\ &(p_{x+1} - p_x)dt + \frac{\beta\gamma_N}2\big(V'(r_{x+1}) + V'(r_{x-1}) - 2V'(r_x)\big)dt\\
    &+ \sqrt{\gamma_N}\big(dB_t^{x-1} - dB_t^x\big),
  \end{aligned}
  \right.
\end{equation*}
where $\{B_\cdot^x;x\in\bT_N\}$ is a system of independent, standard Brownian motions.
Notice that the total momentum $\sum_x p_x$ and the total length $\sum_x r_x$ are the only conserved quantities of the microscopic dynamics.
The conservation law of energy is no longer preserved by $\cS_N$.

\subsection{Stationary states}

The stationary states of $\cL_N$ are given by the family of canonical Gibbs measures indexed by the global momentum $\bar p\in\bR$ and tension $\tau\in\bR$:
\begin{equation*}
  \nu_{\bar p,\tau}^N(dp\,dr) = \bigotimes_{x\in\bT_N} \sqrt{\frac\beta{2\pi}}\exp\left\{-\frac{\beta(p_x - \bar p)^2}2\right\}dp_x \otimes \pi_\tau(dr_x),
\end{equation*}
where the probability measure $\pi_\tau$ is defined as
\begin{equation*}
  \pi_\tau(dr) := \frac1{Z(\tau)}e^{-\beta(V(r) - \tau r)}dr, \quad Z(\tau) := \int_\bR e^{-\beta(V(r) - \tau r)}dr.
\end{equation*}
Observe that the dependence on $\beta$ is omitted, since it is fixed hereafter.
It is easy to see that $\cA_N$, $\cS_N$ are respectively anti-symmetric and symmetric with respect to the Gibbs states.
Moreover, for all smooth functions $f$, $g$ on $\Omega_N$,
\begin{equation*}
  \int_{\Omega_N} f\big(\cS_Ng\big)\,d\nu_{\bar p,\tau}^N = -\frac12\int_{\Omega_N} \sum_{x\in\bT_N} \big(\cY_xf\big)\big(\cY_xg\big)d\nu_{\bar p,\tau}^N.
\end{equation*}

Define the Gibbs potential $G=G(\tau)$ for $\tau \in \bR$ and the free energy $F=F(r)$ for $r \in \bR$ by the following Legendre transform
\begin{equation*}
  G(\tau) := \frac1\beta\log Z(\tau), \quad F (r) := \sup_{\tau\in\bR} \big\{\tau r - G (\tau)\big\}. 
\end{equation*}
The average length $\bar r=\bar r(\tau)$ and equilibrium tension $\bst=\bst(r)$ are then given by the convex conjugate variables
\begin{equation*}
  \bar r (\tau) := E_{\pi_{\tau}} [r] = G' (\tau), \quad \bst (r) := F' (r). 
\end{equation*}

\subsection{Equilibrium perturbation}

As illustrated in Section \ref{sec:intr}, we fix $(\fp_*,\fr_*)\in\bR^2$ and consider the distribution $\mu_{N,0}$ associated to the profile $(\fp_N^\ini,\fr_N^\ini)$ given by
\begin{align*}
  \begin{pmatrix}
    \fp_N^\ini\\ \fr_N^\ini
  \end{pmatrix}
  :=
  \begin{pmatrix}
    \fp_*\\ \fr_*
  \end{pmatrix}
  +N^{-\alpha}\sum_{j=\pm}\sigma_j^\ini\bsu_j, \quad \bsu_\pm:=
  \begin{pmatrix}
    \pm\sqrt{\bst'(\fr_*)}\\
    1
  \end{pmatrix},
\end{align*}
where $\alpha>0$ and $\sigma_\pm^\ini \in C^\infty(\bT)$.
More precisely, for any  $\vf \in C(\bT)$ and $\ve>0$,
\begin{equation*}
  \lim_{N\to\infty} \mu_{N,0} \bigg\{ \bigg| \frac1{N^{1-\alpha}}\sum_{x\in\bT_N}
  \begin{pmatrix}
    p_x-\fp_*\\
    r_x-\fr_*
  \end{pmatrix}
  \vf \left( \frac xN \right) - \sum_{j=\pm} \bsu_j\int_{\bT} \sigma_j^\ini(u)\vf(u)du \bigg| > \ve \bigg\} = 0.
\end{equation*}
In addition, we require that (\emph{cf.} \eqref{eq:assp-chain})
\begin{align*}
  \int_\bT \sigma_-(u)du=\int_\bT \sigma_+(u)du=0.
\end{align*}
For $0<\kappa\le\alpha$, denote by $\{\eta(t);t\ge0\}$ the Markov process generated by $N^{1+\kappa}\cL_N$ and the initial distribution $\mu_{N,0}$.
As usual, we use the notation $\mu_{N,t}$ for the distribution of $\eta(t)$ on $\Omega_N$.

For $(t,u)\in[0,T]\times\bT$, define $(\fp,\fr)=(\fp,\fr)(t,u)$ by
\begin{equation*}
  \begin{pmatrix}
    \fp_N\\ \fr_N
  \end{pmatrix}
  :=
  \begin{pmatrix}
    \fp_*\\ \fr_*
  \end{pmatrix}
  +N^{-\alpha}\sum_{j=\pm} \sigma_j\big(N^{\kappa-\alpha}t,u+j N^\kappa\sqrt{\bst'(\fr_*)}t\big)\bsu_j,
\end{equation*}
where $\sigma_-$, $\sigma_+$ solve the decoupled system of Burgers equations \eqref{eq:chain-burgers}.
Denote by $\nu_{N,t}$ the slowly varying product measure
\begin{align*}
  \nu_{N,t} (d\eta)=  \bigotimes_{x \in \T_N^d} \nu_{\fp_x^N,\bst_x^N}^1 (d\eta_x), \quad \big(\fp_x^N,\bst_x^N\big):=\big(\fp_N,\bst(\fr_N)\big) \left( t,\frac xN \right).
\end{align*}
Let $f_{N,t}$ be the Radon--Nikodym derivative $d\mu_{N,t}/d\nu_{N,t}$ and recall the relative entropy
\begin{align}
\label{eq:chain-rel-ent}
  H_N (t) = H (f_{N,t}; \nu_{N,t}) := \int_{\Omega_N} f_{N,t}\log f_{N,t} d\nu_{N,t}.
\end{align}
Recall that our argument relies on the smoothness of $\sigma_\pm$, hence we require that $t\in[0,T]$ for (\romannum1) any $T<T_\mathrm{shock}$, the first time when shock appears in the entropy solution to \eqref{eq:chain-burgers} if $\kappa=\alpha$ and (\romannum2) any $T>0$ if $\kappa<\alpha$.
The first result is stated below, \emph{cf.} the case $d=1$ in Theorem \ref{thm1}.

\begin{theorem}
\label{thm:chain-rel-ent}
Suppose that $H_N(0) = o(N^{1-2\alpha})$ and $\alpha\in(0,1/2)$.
If $\kappa\in(0,\alpha]\cap(0,(1-2\alpha)/3)$ and $N^{5\kappa+4\alpha-1} \ll \gamma_N \ll N^{1-\kappa}$, then $H_N(t) = o (N^{1-2\alpha})$ for any $t\in[0,T]$.
\end{theorem}

Theorem \ref{thm:chain-rel-ent} is proved in Section \ref{sec:chain-rel-ent}.
With Theorem \ref{thm:chain-rel-ent} and exactly the same argument used in the proof of Theorem \ref{coro1}, we obtain the equilibrium perturbation.

\begin{theorem}
\label{thm:chain-equi-pert}
Under the assumption of Theorem \ref{thm:chain-rel-ent},
\begin{equation*}
  \begin{aligned}
    \lim_{N\to\infty} \mu_{N,t} \bigg\{ \bigg| \frac1{N^{1-\alpha}}\sum_{x\in\bT_N}
    \left( \frac{r_x-\fr_*}2-\frac{p_x-\fp_*}{2\sqrt{\bst'(\fr_*)}} \right)
    \vf \left( \frac xN-N^\kappa\sqrt{\bst'(\fr_*)}t \right)\\
    - \int_\bT \sigma_-^{(\alpha,\kappa)}(t,u)\vf(u)du \bigg| >\ve \bigg\} = 0,\\
    \lim_{N\to\infty} \mu_{N,t} \bigg\{ \bigg| \frac1{N^{1-\alpha}}\sum_{x\in\bT_N}
    \left( \frac{r_x-\fr_*}2+\frac{p_x-\fp_*}{2\sqrt{\bst'(\fr_*)}} \right)
    \vf \left( \frac xN+N^\kappa\sqrt{\bst'(\fr_*)}t \right)\\
    - \int_\bT \sigma_+^{(\alpha,\kappa)}(t,u)\vf(u)du \bigg| >\ve \bigg\} = 0,
  \end{aligned}
\end{equation*}
for any $t\in[0,T]$, $\vf \in C(\bT)$ and $\ve>0$, where
\begin{equation*}
  \sigma_\pm^{(\alpha,\kappa)}(t,u):=
  \begin{cases}
    \sigma_\pm(t,u), &\text{if }\,0<\alpha<\frac15,\,\kappa=\alpha,\\
    \sigma_\pm^\ini(u), &\text{if }\,0<\alpha<\frac12,\,0<\kappa<\min \left\{ \alpha,\frac{1-2\alpha}3 \right\}.
  \end{cases}
\end{equation*}
\end{theorem}

\section{Relative entropy for the oscillator chain}
\label{sec:chain-rel-ent}

We have seen in \eqref{eq:correction} that, for a system of $2$ conservation laws, the non-resonant system of perturbations \eqref{eq:chain-burgers} requires proper second order correction terms.
Hence, we choose the modified profile $(\tilde\fp_N,\tilde\fr_N)(t,u)$ given by
\begin{align*}
  \begin{pmatrix}
    \fp_N\\ \fr_N
  \end{pmatrix}
  +N^{-2\alpha}\sum_{j=\pm} \tilde\sigma_j\big(N^{\kappa-\alpha}t,u-N^\kappa\sqrt{\bst'(\fr_*)}t, u+N^\kappa\sqrt{\bst'(\fr_*)}t\big)\bsu_j,
\end{align*}
where for $(t,u_-,u_+)\in[0,T]\times\bR^2$,
\begin{equation}
\label{eq:chain-correction}
  \begin{aligned}
    \tilde\sigma_-(t,u_-&,u_+):=-\frac{\bst''(\fr_*)}{8\bst'(\fr_*)}\sigma_+^2(t,u_+)\\
    &-\frac{\bst''(\fr_*)}{4\bst'(\fr_*)}\big[\partial_u\sigma_-(t,u_-)\Sigma_+(t,u_+)+\sigma_-(t,u_-)\sigma_+(t,u_+)\big],\\
    \tilde\sigma_+(t,u_-&,u_+):=-\frac{\bst''(\fr_*)}{8\bst'(\fr_*)}\sigma_-^2(t,u_-)\\
    &-\frac{\bst''(\fr_*)}{4\bst'(\fr_*)}\big[\sigma_-(t,u_-)\sigma_+(t,u_+)+\Sigma_-(t,u_-)\partial_u\sigma_+(t,u_+)\big].
  \end{aligned}
\end{equation}
Here $\Sigma_\pm=\int_0^\cdot \sigma_\pm(t,u)du$ are the primitive functions of $\sigma_\pm$.

Let $\tilde\nu_{N,t}$ be the product measure on $\Omega_N$ associated to the profile $(\tilde\fp,\tilde\fr)(t,\cdot)$.
Recall that $\mu_{N,t}$ is the distribution of the dynamics $\eta(t)$.
Denote by $\tilde f_{N,t}$ Radon--Nikodym derivative of $\mu_{N,t}$ with respect to $\tilde\nu_{N,t}$ and by $\widetilde H_N(t)$ the corresponding relative entropy.
Also recall the relative entropy $H_N(t)$ in \eqref{eq:chain-rel-ent}.
The following lemma is straightforward.

\begin{lemma} For any $t \geq 0$, $H_N(t) = o(N^{1-2\alpha})$ if and only if $\widetilde{H}_N (t) = o(N^{1-2\alpha})$.
\end{lemma}

\begin{proof}
Assume first that $\widetilde H_N(t) = o(N^{1-2\alpha})$.
By the definition of relative entropy,
\begin{align*}
  H_N(t) - \widetilde H_N(t) = \int_{\Omega_N} g_{N,t}\,d\mu_{N,t}, \quad g_{N,t}:=\log \left( \frac{d\tilde\nu_{N,t}}{d\nu_{N,t}} \right).
\end{align*}
Direct calculations show that $g_{N,t}(\eta)$ is equal to
\begin{equation*}
  \begin{aligned}
    \beta\sum_{x\in\bT_N} \Big[(\tilde\fp_x^N-\fp_x^N)(p_x-\tilde\fp_x^N)+(\tilde\bst_x^N-\bst_x^N)(r_x-\tilde\fr_x^N)\Big]\\ 
    -\beta\sum_{x\in\bT_N} \left[ \frac{(\tilde\fp_x^N-\fp_x^N)^2}2+G(\tilde\bst_x^N)-G(\bst_x^N)-\tilde\fr_x^N(\tilde\bst_x^N-\bst_x^N) \right]
  \end{aligned}
\end{equation*}
for each $\eta=(p_x,r_x)_{x\in\bT_N}\in\Omega_N$.
The second line is bounded from above by $\mathcal O(N^{1-4\alpha})$.
Meanwhile, by the entropy inequality, the integral of the first line is bounded from above by
\begin{align*}
  &\widetilde H_N (t) + \sum _{x\in\bT_N} \log\int_{\Omega_N} \exp \left\{ \beta
  \begin{pmatrix}
    \tilde\fp_x^N-\fp_x^N\\ \tilde\bst_x^N-\bst_x^N
  \end{pmatrix}
  \cdot
  \begin{pmatrix}
    p_x-\tilde\fp_x^N\\ r_x-\tilde\fr_x^N
  \end{pmatrix}
  \right\} d\tilde\nu_{N,t}\\
  \le\;&\widetilde H_N (t) + \frac{\beta^2}2\sum_{x\in\bT_N} \left(\big|\tilde\fp_x^N-\fp_x^N\big|^2+c_x^2\big|\tilde\fr_x^N-\fr_x^N\big|^2\right),
\end{align*}
where $c_x^2$ is the sub-Gaussian order of $r_x-\tilde\fr_x^N$ under $\tilde\nu_{N,t}$ given by Lemma \ref{lem:general subgaussian}.
Observe that the last term is bounded by $\mathcal O(N^{1-4\alpha})$ and hence $H_N(t) \le 2\widetilde H_N(t)+CN^{1-4\alpha} = o(N^{1-2\alpha})$.
The inverse assertion follows similarly.
\end{proof}

Let $\psi_{N,t}$ be the density function of $\tilde\nu_{N,t}$ with respect to some fixed reference measure $\nu_{\bar p,\tau}^N$.
Without loss of generality, we can choose $(\bar p,\tau)=(0,0)$.
Define the Dirichlet form (\emph{cf.} \eqref{ep:dirichlet})
\begin{align*}
  D_N\big(\tilde f_{N,t};\nu_{N,t}\big) := \int_{\Omega_N} \sum_{x\in\bT_N} \left( \cY_x\sqrt{\tilde f_{N,t}} \right)^2 d\nu_{N,t}.
\end{align*}
Standard manipulation gives
\begin{align*}
  \frac d{dt}\widetilde H_N(t)=\int_{\Omega_N} \tilde f_{N,t} \left( N^{1+\kappa}\cL_N \log\tilde f_{N,t}-\frac d{dt}\log\psi_{N,t} \right) d\tilde\nu_{N,t}.
\end{align*}
Recalling the definition of $\cL_N$ in \eqref{eq:chain-gen}, we have
\begin{align*}
  \tilde f_{N,t}\cL_N\log\tilde f_{N,t} &= \cL_N\tilde f_{N,t}-\frac{\gamma_N}2\sum_{x\in\bT_N} \tilde f_{N,t}^{-1}\big(\cY_x\tilde f_{N,t}\big)^2\\
  &= \cL_N\tilde f_{N,t} -2\gamma_N\sum_{x\in\bT_N} \left( \cY_x\sqrt{\tilde f_{N,t}} \right)^2.
\end{align*}
Integrate it with respect to $\tilde\nu_{N,t}$ and notice that
\begin{equation*}
  \begin{aligned}
    \int_{\Omega_N} \cA_N\tilde f_{N,t}d\tilde\nu_{N,t} &= -\int_{\Omega_N} \tilde f_{N,t}\cA_N\psi_{N,t}\frac{d\tilde\nu_{N,t}}{\psi_{N,t}} = -\int_{\Omega_N} \tilde f_{N,t}\cA_N\log\psi_{N,t}\,d\tilde\nu_{N,t}, \\
    \int_{\Omega_N} \cS_N\tilde f_{N,t}d\tilde\nu_{N,t} &= -\frac12\sum_{x\in\bT_N} \int_{\Omega_N} \cY_x\tilde f_{N,t} \cdot \cY_x\psi_{N,t}\frac{d\tilde\nu_{N,t}}{\psi_{N,t}} \\
    &\le D_N \big(\tilde f_{N,t};\tilde\nu_{N,t}\big) + \frac14\int_{\Omega_N} \tilde f_{N,t}\sum_{x\in\bT_N} \big(\cY_x\log\psi_{N,t}\big)^2d\tilde\nu_{N,t},
  \end{aligned}
\end{equation*}
where the last estimate follows from \CS inequality.
Hence,
\begin{equation*}
  \begin{aligned}
	\frac d{dt}\widetilde{H}_N(t) \le -\gamma_NN^{1+\kappa}D\big(\tilde f_{N,t};\tilde\nu_{N,t}\big)+\int_{\Omega_N} \tilde f_{N,t}J_{N,t}\,d\tilde\nu_{N,t}\\
	+\frac{\gamma_NN^{1+\kappa}}4\int_{\Omega_N} \tilde f_{N,t}\sum_{x\in\bT_N} \big(\cY_x\log\psi_{N,t}\big)^2d\tilde\nu_{N,t},
	\end{aligned}
\end{equation*}
where $J_{N,t}:=(-N^{1+\kappa}\cA_N-d/dt)\log\psi_{N,t}$.

With the choice $(\bar p,\tau)=(0,0)$, $\log\psi_{N,t}=\log(d\tilde\nu_{N,t}/\nu_{0,0}^N)$ reads
\begin{align*}
  \log\psi_{N,t}=\beta\sum_{x\in\bT_N} \left[ \tilde\fp_x^Np_x-\frac{(\tilde\fp_x^N)^2}2+\tilde{\bst}_x^Nr_x-G\big(\tilde{\bst}_x^N\big)+G(0) \right].
\end{align*}
Elementary computation then shows that
\begin{equation*}
  \begin{aligned}
    \sum_{x\in\bT_N} \big(\cY_x\log\psi_{N,t}\big)^2 &= \beta^2\sum_{x\in\bT_N} \big(\tilde\bst_{x+1}^N-\tilde\bst_x^N\big)^2,\\
	-\cA_N\log\psi_{N,t} &= \beta\sum_{x\in\bT_N} \begin{pmatrix}\tilde{\bst}_{x+1}^N - \tilde{\bst}_x^N \\ \tilde{\fp}_x^N - \tilde{\fp}_{x-1}^N\end{pmatrix} \cdot \begin{pmatrix}p_x - \tilde{\fp}_x^N \\ V'(r_x) - \tilde{\bst}_x^N\end{pmatrix}, \\
	-\frac d{dt}\log\psi_{N,t} &= -\beta\sum_{x\in\bT_N} \frac d{dt}\begin{pmatrix}\tilde{\fp}_x^N \\ \tilde{\fr}_x^N\end{pmatrix} \cdot \begin{pmatrix}p_x -\tilde{ \fp}_x^N \\ \bst' (\tilde{\fr}_x^N)(r_x - \tilde{\fr}_x^N)\end{pmatrix}. 
\end{aligned}
\end{equation*}
Therefore, we finally obtain the following inequality, \emph{cf.} \eqref{eqn:relatEnt}:
\begin{align}
\label{eq:chain-ent-ineq}
  \frac d{dt} \widetilde H_N(t) \le -\gamma_NN^{1+\kappa}D\big(\tilde f_{N,t};\tilde\nu_{N,t}\big) + \beta\big(\cR_{N,t} + \cE_{N,t}\big),
\end{align}
where
\begin{equation*}
  \begin{aligned}
    \cR_{N,t} := \int_{\Omega_N} &\tilde f_{N,t}\sum_{x\in\bT_N} \frac{d\tilde\fr_x^N}{dt}\big[V'(r_x) - \tilde\bst_x^N - \bst'(\tilde\fr_x^N)(r_x -\tilde\fr_x^N)\big]d\tilde\nu_{N,t};\\
    \cE_{N,t} := \int_{\Omega_N} &\tilde f_{N,t}\sum_{x\in\bT_N} \left[ \frac{\gamma_NN^{1+\kappa}\beta}4\big(\nabla\tilde\bst_x^N\big)^2 + \ep_x^N \cdot
    \begin{pmatrix}
      p_x - \tilde{\fp}_x^N \\ V'(r_x) - \tilde{\bst}_x^N
    \end{pmatrix}
    \right] d\tilde\nu_{N,t}.
  \end{aligned}
\end{equation*}
Here we use the abbreviation $\nabla f_x=f_{x+1}-f_x$ and define
\begin{equation*}
    \ep_x^N := N^{1+\kappa}
    \begin{pmatrix}
      \nabla\tilde\bst_x^N \\ \nabla\tilde\fp_{x-1}^N
    \end{pmatrix}
    -\frac d{dt}
    \begin{pmatrix}
      \tilde\fp_x^N \\ \tilde\fr_x^N
    \end{pmatrix}.
\end{equation*}

\subsection{Proofs of Theorem \ref{thm:chain-rel-ent} and \ref{thm:chain-equi-pert}}

Theorem \ref{thm:chain-rel-ent} follows from the inequality \eqref{eq:chain-ent-ineq}, Propositions \ref{prop:chain-E} and \ref{prop:chain-R} below.

\begin{proposition}
\label{prop:chain-E}
There exists a constant $C$ independent of $N$, such that
\begin{align*}
  \cE_{N,t}\le \widetilde H_N(t)+C\big(\gamma_NN^{\kappa-2\alpha}+N^{-1+2\kappa-2\alpha}+N^{1+2\kappa-6\alpha}\big).
\end{align*}
\end{proposition}

\begin{proof}
Recall that the profile $(\tilde\fp_N,\tilde\fr_N)$ is explicitly given by
\begin{align*}
  \begin{pmatrix}
    \tilde\fp_N \\ \tilde\fr_N
  \end{pmatrix}
  (t,u) =
  \begin{pmatrix}
    \fp_* \\ \fr_*
  \end{pmatrix}
  +\sum_{j=\pm} \left[ \frac{\sigma_j(s,u_j)}{N^\alpha}+\frac{\tilde\sigma_j(s,u_-,u_+)}{N^{2\alpha}} \right] \bsu_j,
\end{align*}
where $\sigma_\pm=\sigma_\pm(s,u)$, $\tilde\sigma_\pm=\tilde\sigma(s,u_1,u_2)$ are smooth functions given respectively by \eqref{eq:chain-burgers} and \eqref{eq:chain-correction}, $s=N^{\kappa-\alpha}t$, $\bsu_-=(-\sqrt{\bst'(\fr_*)},1)'$, $\bsu_+=(\sqrt{\bst'(\fr_*)},1)'$ and
\begin{align*}
  u_-=u-N^\kappa\sqrt{\bst'(\fr_*)}t, \quad u_+=u+N^\kappa\sqrt{\bst'(\fr_*)}t.
\end{align*}

For the first term in $\cE_{N,t}$, since $\nabla\tilde\bst_x^N \le N^{-1}\sup_x |\partial_x\bst(\tilde\fr(t,x))| \le CN^{-1-\alpha}$,
\begin{align*}
  \frac{\gamma_NN^{1+\kappa}\beta^2}4\int_{\Omega_N} \tilde f_{N,t}\sum_{x\in\bT_N} \big(\nabla\tilde\bst_x^N\big)^2d\tilde\nu_{N,t} \le C\gamma_NN^{\kappa-2\alpha}.
\end{align*}

We focus on the second term in $\cE_{N,t}$.
By Taylor's expansion,
\begin{align*}
  \begin{pmatrix}
    \nabla\tilde\bst_x^N \\ \nabla\tilde\fp_{x-1}^N
  \end{pmatrix}
  =\frac1N\frac\partial{\partial u}
  \begin{pmatrix}
    \bst(\tilde\fr_N)\\
    \tilde\fp_N
  \end{pmatrix}
  \left( t,\frac xN \right) + \mathcal O\big(N^{-2-\alpha}\big).
\end{align*}
Expanding the function $(\fp,\fr)\mapsto(\bst(\fr),\fp)$ at $(\fp_*,\fr_*)$ up to the second order,
\begin{multline*}
    \begin{pmatrix}
	\nabla\tilde\bst_x^N \\ \nabla\tilde\fp_{x-1}^N
\end{pmatrix}
=A\sum_{j=\pm} \left[ \frac{\partial_u\sigma_j(s,u_j)}{N^{1+\alpha}}+\frac{(\partial_{u_1}+\partial_{u_2})\tilde\sigma_j(s,u_-,u_+)}{N^{1+2\alpha}} \right] \bsu_j\\
+\frac{\partial_u[(\sigma_-(s,u_-)+\sigma_+(s,u_+))^2]}{2N^{1+2\alpha}}\mathbf b + \mathcal O\big(N^{-2-\alpha}+N^{-1-3\alpha}\big),
\end{multline*}
where $A=\left[\begin{smallmatrix}0 &\bst'(\fr_*)\\1 &0\end{smallmatrix}\right]$ and $\mathbf b=(\bst''(\fr_*),0)'$.
Meanwhile,
\begin{multline*}
    \frac d{dt}
\begin{pmatrix}
	\tilde{\fp}_x^N \\ \tilde{\fr}_x^N
\end{pmatrix}
=\sum_{j=\pm} \left[ \frac{\partial_s\sigma_j(s,u_j)}{N^{-\kappa+2\alpha}}+\frac{j\sqrt{\bst'(\fr_*)}\partial_u\sigma_j(s,u_j)}{N^{-\kappa+\alpha}} \right] \bsu_j\\
+\frac{\sqrt{\bst'(\fr_*)}}{N^{-\kappa+2\alpha}}\sum_{j=\pm} (-\partial_{u_1}+\partial_{u_2})\tilde\sigma_j(s,u_-,u_+)\bsu_j + \mathcal O\big(N^{\kappa-3\alpha}\big).
\end{multline*}
Noticing that $A\bsu_\pm=\pm\sqrt{\bst'(\fr_*)}\bsu_\pm$ for $j=\pm$,
\begin{equation}\label{chain1}
  \begin{aligned}
    \ep_x^N=\;&\frac1{N^{-\kappa+2\alpha}} \bigg\{ \big[-\partial_s\sigma_-(s,u_-)-2\sqrt{\bst'(\fr_*)}\partial_{u_2}\tilde\sigma_-(s,u_-,u_+)\big]\bsu_-\\
    &+\big[-\partial_s\sigma_+(s,u_+)+2\sqrt{\bst'(\fr_*)}\partial_{u_1}\tilde\sigma_+(s,u_-,u_+)\big]\bsu_+\\
    &+\frac{\partial_u[(\sigma_-(s,u_-)+\sigma_+(s,u_+))^2]}2\mathbf b \bigg\} + \mathcal O\big(N^{-1+\kappa-\alpha}+N^{\kappa-3\alpha}\big).
  \end{aligned}
\end{equation}
We show that the terms in the first bracket vanishes.
From \eqref{eq:chain-correction},
\begin{align*}
  \partial_{u_2}\tilde\sigma_-(s,u_-,u_+)=-\frac{\bst''(\fr_*)}{8\bst'(\fr_*)} \partial_u\big[\sigma_+^2(t,u_+)+2\sigma_-(s,u_-)\sigma_+(s,u_+)\big],\\
  \partial_{u_1}\tilde\sigma_+(s,u_-,u_+)=-\frac{\bst''(\fr_*)}{8\bst'(\fr_*)} \partial_u\big[\sigma_-^2(t,u_+)+2\sigma_-(s,u_-)\sigma_+(s,u_+)\big].
\end{align*}
Therefore, with the identity $\bst''(\fr_*)(\bsu_--\bsu_+)=-2\sqrt{\bst'(\fr_*)}\mathbf b$,
\begin{multline*}
    -2\sqrt{\bst'(\fr_*)}\partial_{u_2}\tilde\sigma_-(s,u_-,u_+)\bsu_-+2\sqrt{\bst'(\fr_*)}\partial_{u_1}\tilde\sigma_+(s,u_-,u_+)\bsu_+\\
=\frac{\bst''(\fr_*)}{4\sqrt{\bst'(\fr_*)}}\partial_u\big[\sigma_+^2(s,u_+)\bsu_--\sigma_-^2(s,u_-)\bsu_+\big]-\partial_u\big[\sigma_-(s,u_-)\sigma_+(s,u_+)\big]\mathbf b.
\end{multline*}
The bracket in \eqref{chain1} is then equal to
\begin{multline*}
    -\partial_s\sigma_-(s,u_-)\bsu_- - \left[ \frac{\bst''(\fr_*)}{4\sqrt{\bst'(\fr_*)}}\bsu_+-\frac12\mathbf b \right] \partial_u\sigma_-^2(s,u_-)\\
-\partial_s\sigma_+(s,u_+)\bsu_+ + \left[ \frac{\bst''(\fr_*)}{4\sqrt{\bst'(\fr_*)}}\bsu_-+\frac12\mathbf b \right] \partial_u\sigma_+^2(s,u_+).
\end{multline*}
Observe from \eqref{eq:chain-burgers} that it is identically zero, hence
\begin{align*}
  \ep_x^N \le C\big(N^{-1+\kappa-\alpha}+N^{\kappa-3\alpha}\big).
\end{align*}
The conclusion then follows from the relative entropy inequality.
\end{proof}

\begin{corollary}[Harmonic chain]
When the oscillators perform harmonic interaction, i.e., $V (r) \propto r^2$,
$\cR_{N,t}$ is identically zero.
Hence, Theorems \ref{thm:chain-rel-ent} and \ref{thm:chain-equi-pert} hold autonomously.
\end{corollary}

\begin{proposition}
\label{prop:chain-R}
There exists a constant $C$ independent of $N$, such that
\begin{align*}
  \cR_{N,t} \le \frac{\gamma_NN^{1+\kappa}}\beta D_N \big(\tilde f_{N,t};\tilde\nu_{N,t}\big) + C \left( \widetilde H_N(t) + \gamma_N^{-\frac15}N^{\frac45+\kappa-\frac65\alpha} \right).
\end{align*}
\end{proposition}

\begin{proof}
We follow the proof of \cite[Lemma 3.1]{Xu20}.
To shorten the notation, let
\begin{align*}
  a_x^N:=\partial_t\tilde\fr_N \left( t,\frac xN \right), \quad \phi_x:=V^\prime (r_x)-\bst(\tilde\fr_x^N)-\bst'(\tilde\fr_x^N)\big(r_x-\tilde\fr_x^N\big).
\end{align*}
Observe that
\begin{align*}
  \big|a_x^N\big| \le N^{-\alpha} \sum_{j=\pm} jN^\kappa\partial_u\sigma_j|\bsu_j| + \mathcal O\big(N^{\kappa-2\alpha}\big) \le CN^{\kappa-\alpha},\\
  \big|a_x^N-a_y^N\big| \le C|x-y|N^{-1+\kappa-\alpha}.
\end{align*}
Fix a mesoscopic scale $\ell = \ell(N) \ll N$ and define
\begin{align*}
  \phi_x^\ell:=\frac1\ell\sum_{y=0}^{\ell-1} \phi_{x+y}, \quad \bar\phi_x^\ell:=E_{\tilde\nu_{N,t}} \left[ \phi_x^\ell~\Big|~\sum_{y=0}^{\ell-1} r_{x+y} \right].
\end{align*}
First, one can replace $\phi_x$ in $ \cR_{N,t}$ by its block average.
The error is bounded by
\begin{align*}
  \cR_{N,t} - \int_{\Omega_N} \tilde f_{N,t}\sum_{x\in\bT_N} a_x^N\phi_x^\ell\,d\tilde\nu_{N,t} = \int_{\Omega_N} \tilde f_{N,t}\sum_{x\in\bT_N} \phi_x\sum_{y=0}^{\ell-1} \frac{a_x^N-a_{x-y}^N}\ell d\tilde\nu_{N,t}.
\end{align*}
Note that $|V'(r_x)-\bst'(\tilde\fr_x^N)r_x|<c|r|$.
By Lemma \ref{lem:general subgaussian}, $\phi_x$ is sub-Gaussian with uniformly bounded order $c_x^2$.
Applying relative entropy inequality, the left-hand side is bounded from above by
\begin{align*}
  &\widetilde H_N(t)
  +\sum_{x\in\bT_N} \log\int_{\Omega_N} \exp \left\{\phi_x\sum_{y=0}^{\ell-1} \frac{a_x^N-a_{x-y}^N}\ell\right\} d\tilde\nu_{N,t} \\
  \le\,&\widetilde H_N(t) + \sum_{x\in\bT_N} \frac{c_x^2}2\bigg|\sum_{y=0}^{\ell-1} \frac{a_x^N-a_{x-y}^N}\ell\bigg|^2 \le \widetilde H_N(t) + C\ell^2N^{-1+2\kappa-2\alpha},
\end{align*}
Hence, we obtain the estimate
\begin{equation}
\label{eq:chain-1}
  \cR_{N,t} - \int_{\Omega_N} \tilde f_{N,t}\sum_{x\in\bT_N} a_x^N\phi_x^\ell\,d\tilde\nu_{N,t} \le H_N(t) + C\ell^2N^{-1+2\kappa-2\alpha}.
\end{equation}
Next, let $\psi_x^\ell=\psi_x^\ell(r_x,\ldots,r_{x+\ell-1})$ solve the Poisson equation
\begin{align*}
  \sum_{y=0}^{\ell-2} \cY_{t,x+y}^*\cY_{x+y}\psi_x^\ell=\phi_x^\ell-\bar\phi_x^\ell,
\end{align*}
where for each $x$, $\cY_{t,x}^*$ is the adjoint operator of $\cY_x$ with respect to $\tilde\nu_{N,t}$:
\begin{align*}
  \cY_{t,x}^*:=\beta\big[V'(r_{x+1})-V'(r_x)-\bst_{x+1}^N+\bst_x^N\big]-\cY_x.
\end{align*}
Then, for each $x\in\bT_N$,
\begin{align*}
  \int_{\Omega_N} \tilde f_{N,t}a_x^N\big(\phi_x^\ell-\bar\phi_x^\ell\big)d\tilde\nu_{N,t} = \int_{\Omega_N} a_x^N\sum_{y=0}^{\ell-2} \big(\cY_{x+y}\tilde f_{N,t}\big)\big(\cY_{x+y}\psi_x^\ell\big)d\tilde\nu_{N,t}.
\end{align*}
Summing up for $x\in\bT_N$ and applying \CS inequality,
\begin{multline*}
  \int_{\Omega_N} \tilde f_{N,t}\sum_{x\in\bT_N} a_x^N\big(\phi_x^\ell-\bar\phi_x^\ell\big)d\tilde\nu_{N,t} \le \frac{\gamma_NN^{1+\kappa}}{\beta(\ell-1)}\int_{\Omega_N} \sum_{x\in\bT_N} \sum_{y=0}^{\ell-2} \frac{(\cY_{x+y}\tilde f_{N,t})^2}{4\tilde f_{N,t}}d\tilde\nu_{N,t}\\
+\frac{\beta(\ell-1)}{\gamma_NN^{1+\kappa}}\int_{\Omega_N} \sum_{x\in\bT_N} (a_x^N)^2\sum_{y=0}^{\ell-2} \tilde f_{N,t}\big(\cY_{x+y}\psi_x^\ell\big)^2d\tilde\nu_{N,t}.
\end{multline*}
The first term in the right-hand side gives $\beta^{-1}\gamma_NN^{1+\kappa}D_N(\tilde f_{N,t};\tilde\nu_{N,t})$.
To estimate the second term, note that since $V''$ is bounded, we can apply the gradient estimate for Poisson equation (see \cite[Proposition 9.1]{Xu20}, also \emph{cf.} \cite[Theorem 1.1]{Wu09}) to obtain a constant $C$ independent of $x$ or $\ell$, such that
\begin{align*}
  \sum_{y=0}^{\ell-2} \big(\cY_{x+y}\psi_x^\ell\big)^2 \le C\ell^4\sup_{(r_x,\ldots,r_{x+\ell-1})}\sum_{y=0}^{\ell-2} \big(\cY_{x+y}[\phi_x^\ell-\bar\phi_x^\ell]\big)^2.
\end{align*}
Since $\bar\phi_x^\ell$ is a function of $r_x+\ldots+r_{x+\ell-1}$, $\cY_{x+y}\bar\phi_x^\ell\equiv0$.
Also,
\begin{align*}
  \big|\cY_{x+y}\phi_x^\ell\big|=\frac1\ell\big|V'(r_{x+y+1})-V'(r_{x+y})-\bst'(\fr_{x+y+1}^N)+\bst'(\fr_{x+y}^N)\big|\le\frac C\ell.
\end{align*}
Direct computation then shows that
\[  \int_{\Omega_N} \sum_{x\in\bT_N} (a_x^N)^2\sum_{y=0}^{\ell-2} \tilde f_{N,t}\big(\cY_{x+y}\psi_x^\ell\big)^2d\tilde\nu_{N,t} \le C\sum_{x\in\bT_N} (a_x^N)^2\ell^3 \le C'\ell^3N^{1+2\kappa-2\alpha}.\]
Therefore, we obtain the estimate
\begin{align}
\label{eq:chain-2}
  \int_{\Omega_N} \tilde f_{N,t}\sum_{x\in\bT_N} a_x^N\big(\phi_x^\ell-\bar\phi_x^\ell\big)d\tilde\nu_{N,t} \le \frac{\gamma_NN^{1+\kappa}}\beta D_N\big(\tilde f_{N,t};\tilde\nu_{N,t}\big) + C_\beta\frac{\ell^4N^{\kappa-2\alpha}}{\gamma_N}.
\end{align}
For the space variance of $\bar\phi_x^\ell$, relative entropy inequality reads
\[  \int_{\Omega_N} \tilde f_{N,t}\sum_{x\in\bT_N} a_x^N\bar\phi_x^\ell\,d\tilde\nu_{N,t} \le \frac{N^{\kappa-\alpha}}\delta \left[ \widetilde H_N(t) + \frac1\ell\sum_x \log \int_{\Omega_N} e^{\delta N^{\alpha-\kappa}a_x^N\ell\bar\phi_x^\ell}d\tilde\nu_{N,t} \right],\]
for any $\delta>0$.
Observe that the extra factor $\ell$ in the last term above is because that $\bar\phi_x^\ell$ is independent of $\bar\phi_y^\ell$ for any $|x-y|\ge\ell$, see, e.g., \cite[Lemma F.12]{jaram18nonequilireaction} and \cite[Lemma D.3]{Xu20}.
Recall that $N^{\alpha-\kappa}a_x^N$ is bounded.
To treat the exponential moment in above, we apply the equivalence of inhomogeneous ensembles \cite[Proposition 8.3]{Xu20} to obtain that for $\ell=o(N^{\frac23})$ and $\delta$ sufficiently small but fixed,
\[  \log\int_{\Omega_N} \exp\big\{\delta (N^{\alpha-\kappa}a_x^N)\ell\bar\phi_x^\ell\big\}d\tilde\nu_{N,t} \le C.\]
Since $N^{\kappa-\alpha}\le\mathcal O(1)$,
\begin{align}
\label{eq:chain-3}
  \int_{\Omega_N} \tilde f_{N,t}\sum_{x\in\bT_N} a_x^N\bar\phi_x^\ell\,d\tilde\nu_{N,t} \le C\big(\widetilde H_N(t) + \ell^{-1}N^{1+\kappa-\alpha}\big).
\end{align}

Finally, the proof is concluded by choosing $\ell(N)=\gamma_N^{\frac15}N^{\frac{1+\alpha}5}$ and adding up the estimates \eqref{eq:chain-1}, \eqref{eq:chain-2} and \eqref{eq:chain-3}.
\end{proof}

\begin{proof}[Proof of Theorem \ref{thm:chain-rel-ent}]
By \eqref{eq:chain-ent-ineq}, Proposition \ref{prop:chain-E} and \ref{prop:chain-R},
\[  \frac d{dt}\widetilde H_N(t) \le C \Big( \widetilde H_N(t)+\gamma_NN^{\kappa-2\alpha}+N^{-1+2\kappa-2\alpha}
+N^{1+2\kappa-6\alpha}+\gamma_N^{-\frac15}N^{\frac45+\kappa-\frac65\alpha} \Big).\]
Since $\kappa\le\alpha<1/2$, $N^{-1+2\kappa-2\alpha}+N^{1+2\kappa-6\alpha}=o(N^{1-2\alpha})$, so that
\[  \frac d{dt}\widetilde H_N(t) \le C\widetilde H_N(t)+C \left[ \gamma_NN^{\kappa-1}+\big(\gamma_N^{-1}N^{-1+5\kappa+4\alpha}\big)^{\frac15} + o(1) \right] N^{1-2\alpha}.\]
The estimate then follows from the choice of $\gamma_N$ and \Gro inequality.
\end{proof}

Theorem \ref{thm:chain-equi-pert} follows from the exactly same argument as we used in the proof of Theorem \ref{coro1}, hence we omit the proof here.

\appendix

\section{General tools}

In this appendix, we state some model independent tools that is used through the paper.

\subsection{Relative entropy inequality.} In this subsection, we introduce a version of Yau's relative entropy inequality. Let $\{X_t\}_{t \geq 0}$ be a continuous-time  Markov chain on a finite state space $S$, whose infinitesimal generator is  defined as
\[L f (x) = \sum_{y \in S} r(x,y) \big[ f(y) - f(x)\big], \quad x \in S.\]
Above, $r(x,y) \geq 0$ is the rate at which the chain jumps from $x$ to $y$, and $f: S \rightarrow \R$ is any function. Define the \emph{carr{\'e} du champ} operator associated to $L$ as
\[\Gamma f (x) = \sum_{y \in S} r(x,y) \big[ f(y) - f(x)\big]^2, \quad x \in S.\]
Denote by $\mu_t$ the distribution of the process at time $t$ with initial measure $\mu_0$. Let $\{\nu_t\}_{t \geq 0}$  and $\nu$ be a family of measures in $S$ such that  $\nu_t$ is differentiable in time $t$, and $\nu_t(x) > 0$, $\nu(x) > 0$ for any $x \in S$ and any $t \geq 0$. Denote by $f_t$ (respectively $\psi_{t}$) the Radon-Nikodym derivative of $\mu_t$ (respectively $\nu_t$) with respect to $\nu_t$ (respectively $\nu$),
\[f_t (x)= \frac{\mu_t (x)}{\nu_t (x)}, \quad \psi_{t} (x) = \frac{\nu_t(x)}{\nu (x)}, \quad x \in S.\]
Define the relative entropy $H (t)$ as
\[H (t) = H(\mu_t | \nu_t) =  \int  f_t  \log f_t d \nu_t,\]
with the convention $0 \log 0 = 0.$

\begin{lemma}[Yau's relative entropy inequality]\label{lem:relativeEntIne}
	For any $t \geq 0$,
	\[H^\prime (t) \leq - \int \Gamma \sqrt{f_t} d \nu_t+ \int \Big( L_t^* \mathbf{1} - \frac{d}{dt} \log \psi_t \Big) d \mu_t,\]
	where $\mathbf{1}$ is the constant function identical to one, and $L_t^*$ is the adjoint of $L$ with respect to $L^2 (\nu_t)$, which acting on any function $g: S \rightarrow \R$ is given by
	\[L_t^* g (x) = \sum_{y \in S}\Big\{ \frac{\nu_t(y) r(y,x)}{\nu_t (x)} g(y) - r(x,y) g(x) \Big\}, \quad x \in S.\]
\end{lemma}

We refer the readers to \cite[Lemma A.1]{jara2018non} for proof of the above lemma. Compared to the classical Yau's relative entropy inequality (\emph{cf.}\,\cite[Lemma 6.1.4]{klscaling}), an extra term $\int \Gamma \sqrt{f_t} d \nu_t$ is subtracted in the above version.

\subsection{Flow lemma.}\label{subsec:flow}  In this subsection, we state a flow lemma introduced by Jara and Menezes \cite{jara2018non,jaram18nonequilireaction}. For two measures $\pfrak$ and $\qfrak$ on $\Z^d$, we say a function $\phi: \Z^d \times \{e_i\}_{1 \leq i \leq d} \rightarrow \R$ is a \emph{flow connecting $\pfrak$ to $\qfrak$} if for any $z \in \Z^d$,
\[\pfrak (z) - \qfrak (z) = \sum_{i=1}^d (\phi (z,e_i) - \phi(z-e_i,e_i)).\]
The support of the flow $\phi$ is defined as the set of points $\{z,z+e_i\}$ such that $\phi (z,z+e_i) \neq 0.$ Using the summation by parts formula, for any function $f: \Z^d \rightarrow \R$ and any flow $\phi$ connecting $\pfrak$ to $\qfrak$,
\[\sum_{z \in \Z^d} f(z) (\pfrak (z) - \qfrak (z)) = \sum_{i=1}^d \sum_{z \in \Z^d} \phi (z,e_i) (f(z) - f(z+e_i)).\]

For $\ell \geq 1$, let $\pfrak_\ell (\cdot)$ be the uniform measure on $\Lambda_\ell^d = \{0,1,\ldots,\ell-1\}^d$, \emph{i.e.}\;$\pfrak_\ell (x) = \ell^{-d}$ if $x \in \Lambda_\ell^d$ and $=0$ otherwise.  Let $\mathfrak{q}_\ell = \pfrak_\ell \ast \pfrak_\ell$ be the convolution of $\pfrak_\ell$ with itself, 
\[\qfrak_\ell (y) = \sum_{z \in \T_N^d} \pfrak_\ell (z) \pfrak_\ell (y-z), \quad y \in \T_N^d.\]
In the sequel, we shall use $\qfrak_\ell$ to define the spatial average of a random variable over a large box instead of the usual $\pfrak_\ell$. Note that the support of $\qfrak_\ell$ is contained in $\Lambda_{2\ell - 1}^d$.  Let $\delta_0 (\cdot)$ be the Dirac measure concentrated at the origin. For $\ell \geq 1$, denote
\begin{equation}\label{gdl}
	g_d (\ell) = \begin{cases}
		\ell, &\quad d=1,\\	
		\log \ell, &\quad d=2,\\
		1, &\quad d \geq 3.
	\end{cases}
\end{equation}

The following lemma states that we could construct a flow, which connects $\delta_0$ to $\qfrak_\ell$, such that the cost is at most of order $g_d (\ell)$. We refer the readers to \cite{jaram18nonequilireaction,jara2018non} for its proof.

\begin{lemma}[Flow lemma]\label{lem:flow}
	There exists a finite constant $C_0$ such that for any $\ell \geq 1$, there exists a flow $\phi_\ell$ connecting $\delta_0$ to $\qfrak_\ell$ with support in $\Lambda_{2\ell - 1}^d$ such that
	\[\sum_{i=1}^d \sum_{z \in \Z^d} \phi_\ell(z,e_i)^2 \leq C_0 g_d(\ell), \quad \sum_{i=1}^d \sum_{z \in \Z^d} |\phi_\ell(z,e_i)| \leq C_0 \ell.\]
\end{lemma}

\subsection{Concentration inequalities.} In this subsection, we focus on properties of sub-Gaussian random variables. We say a real-valued random variable $X$ is sub-Gaussian of order $\sigma^2$ if
\[\log E [e^{\theta X}] \leq \frac{1}{2} \sigma^2 \theta^2, \quad \forall \theta \in \R.\]
The following lemma gives an equivalent condition for $X$ to be sub-Gaussian by controlling the exponential moment of $X^2$.

\begin{lemma}\label{lem:subGau}
	If the random variable $X$ is sub-Gaussian of order $\sigma^2$, then
	\[E[e^{\gamma X^2} ]\leq 3, \quad \forall\gamma \leq \frac1{4 \sigma^2}.\]
Inversely, if $E [X] = 0$ and $E [e^{\gamma X^2}] \le C$ for some $\gamma > 0$, $C \ge 1$, then $X$ is sub-Gaussian of order $2C\gamma^{-1}$.
\end{lemma}

\begin{proof}
First assume that $X$ is sub-Gaussian of order $\sigma^2$.
	Let $\mathcal{N}$ be a standard normal distribution independent of $X$.
	Denote by $\langle \cdot \rangle$ the expectation with respect to $\mathcal{N}$. Then, for any $\theta \in \R$, 
	\[\log \langle e^{\theta \mathcal{N}}\rangle = \frac{1}{2} \theta^2.\]
	Therefore,
	\[E[e^{\gamma X^2} ] = E [\<e^{\sqrt{2\gamma} X \mathcal{N}} \>] = \<E [e^{\sqrt{2\gamma} X \mathcal{N}}]\>.\]
	Since $X$ is sub-Gaussian of order $\sigma^2$, if $\gamma \leq (4\sigma^2)^{-1}$, the last formula is bounded by
	\[\<e^{\gamma \sigma^2 \mathcal{N}^2}\> = \int_\R \frac{1}{\sqrt{2\pi}} e^{-\frac{1}{2}x^2+\gamma \sigma^2 x^2} dx = \frac{1}{\sqrt{1-2\gamma \sigma^2}} \leq 3.\]
This concludes the proof of the first assertion.

Now assume that $E[X]=0$ and $E[e^{\gamma X^2}]<C$.
For any $s \in \bR$,
$$
  E \big[e^{sX}\big] = 1 + \sum_{k=2}^\infty \frac{E [(sX)^k]}{k!} \le 1 + \frac{s^2}2\sum_{k=0}^\infty \frac{|s|^kE [|X|^{k+2}]}{k!}. 
$$
The summation in the right-hand side is bounded by 
$$
  \frac{s^2}2E \big[X^2e^{|sX|}\big] \le \frac{s^2}2E \left[X^2\exp\left\{\frac {\gamma X^2}2 + \frac{s^2}{2\gamma}\right\}\right] 
$$
for any $c > 0$. 
With the elementary inequality $ye^y \le e^{2y}$, 
$$
  \frac{s^2}2E \left[X^2\exp\left\{\frac {\gamma X^2}2 + \frac{s^2}{2\gamma}\right\}\right] \le \frac{s^2}\gamma\exp\left\{\frac{s^2}{2\gamma}\right\}E \big[e^{\gamma X^2}\big]. 
$$
Hence, by the upper bound of $E [e^{\gamma X^2}]$,
$$
  E \big[e^{sX}\big] \le 1 + \frac{Cs^2}\gamma\exp\left\{\frac{s^2}{2\gamma}\right\} \le \exp\left\{\frac{Cs^2}\gamma\right\}. 
$$
As $s$ is arbitrary, the proof is completed. 
\end{proof}

The constant $3$ in Lemma \ref{lem:subGau} is not optimal and we only need it to be a constant.
The next corollary directly follows from Young's inequality.

\begin{corollary}\label{cor:subGau}
	Let $X_i$ be sub-Gaussian of order $\sigma^2_i$ for $i=1,2$. Then, for any $\gamma \leq (4 \sigma_1 \sigma_2)^{-1}$, 
	\[E[e^{\gamma X_1 X_2}] \leq 3.\]
\end{corollary}

When treating uniformly bounded state space (Section \ref{sec:gasep} and \ref{sec:gasep-rel-ent}), the following lemma is convenient.

\begin{lemma}[Hoeffding's Lemma, {\cite[Lemma 2.2.2]{boucheron2013concentration}}]\label{lem:hoeffding}
	If the random variable $X \in [a,b]$ for some  $a < b$, the $X - E[X]$  is sub-Gaussian of order $(b-a)^2$.
\end{lemma}

The next lemma is necessary to estimate the upper bound of exponential moment of volume in oscillator chain see Section \ref{sec:chain} and \ref{sec:chain-rel-ent}.

\begin{lemma}
\label{lem:general subgaussian}
Let $V \in C(\bR)$ be such that $c_-r^2 \le 2V(r) \le c_+r^2$ with constants $c_\pm>0$. 
For $\tau \in \bR$, recall the probability measure $\pi_\tau$ on $\bR$ given by
$$
  \pi_\tau = \frac1{Z(\tau)}e^{-V(r) + \tau r - G(r)}dr, \quad Z(r) = \int_\bR e^{-V(r) + \tau(r)}dr. 
$$
If $|F(r)| \le c|r|$ with constant $c$, then $F - E_{\pi_\tau} [F]$ is sub-Gaussian of order $C = C(\tau, c, c_\pm)$ under $\pi_\tau$.
\end{lemma}

\begin{proof}
Notice that since $V \le c_+r^2/2$, for all $\tau \in \bR$, 
$$
  Z(\tau) \ge \int_\bR \exp\left\{-\frac{c_+r^2}2 + \tau r\right\} = \sqrt{\frac{2\pi}{c_+}}\exp\left\{\frac{\tau^2}{2c_+}\right\}. 
$$
Similarly, for $t$ such that $0 < t < c_-/(2c^2)$, 
\begin{align*}
  E_{\pi_\tau} \big[\exp(tF^2)\big] &\le \frac1{Z(\tau)}\int_\bR \exp\left\{-\frac{(c_- - 2tc^2)r^2}2 + \tau r\right\}dr \\
  &\le \sqrt{\frac{c_+}{c_- - 2tc^2}}\exp\left\{\frac{\tau^2}2\left(\frac1{c_- - 2tc^2} - \frac1{c_+}\right)\right\}. 
\end{align*}
Denote $F_* = F - E_{\pi_\tau} [F]$. 
By convexity, for all $t \ge 0$, 
$$
  E_{\pi_\tau} \big[\exp(tF_*^2)\big] \le \exp\big(2tE_{\pi_\tau}^2 [F]\big)E_{\pi_\tau} \big[\exp(2tF^2)\big] \le E_{\pi_\tau} \big[\exp(4tF^2)\big]. 
$$
Therefore, we obtain that 
$$
  E_{\pi_\tau} \left[\exp\left(\frac{c_-}{16c^2}F_*^2\right)\right] \le E_{\pi_\tau} \left[\exp\left(\frac{c_-}{4c^2}F^2\right)\right] \le \sqrt{\frac{2c_+}{c_-}}\exp\left\{\frac{\tau^2}{2}\left(\frac2{c_-} - \frac1{c_+}\right)\right\}. 
$$
We can then conclude with Lemma \ref{lem:subGau}.
\end{proof}


\begin{thebibliography}{99}

\bibitem{BHO19}
Bernardin, B., Huveneers, F. and Olla, S.: Hydrodynamic limit for a disordered harmonic chain.
\emph{Comm. Math. Phys.} \textbf{365}(1), (2019), 215--237.

\bibitem{boucheron2013concentration}
Boucheron, S., Lugosi, G. and Massart, P.:
Concentration inequalities: A nonasymptotic theory of independence.
\emph{Oxford university press}, 2013.

\bibitem{EvenO14}
Braxmeier-Even, N. and Olla, S.:
Hydrodynamic limit for an Hamiltonian system with boundary conditions and conservative noise.
\emph{Arch. Ration. Mech. Anal.} \textbf{213}(2), (2014), 561--585.

\bibitem{chang2001equilibrium}
Chang, C. C., Landim, C., and Olla, S.:
Equilibrium fluctuations of asymmetric simple exclusion processes in dimension $d \geq 3$.
\emph{Probab. Theory Relat. Fields} \textbf{119}(3), (2001), 381--409.

\bibitem{DiPerM85}
DiPerna, R. J. and Majda, A.:
The validity of nonlinear geometric optics for weak solutions of conservation laws.
\emph{Comm. Math. Phys.} \textbf{98}, (1985), 313--347.

\bibitem{esposito1994diffusive}
Esposito, R., Marra, R. and Yau, H. T.:
Diffusive limit of asymmetric simple exclusion.
\emph{Rev. Math. Phys.} \textbf{6}(05a), (1994), 1233--1267.

\bibitem{funaki2019motion}
Funaki, T. and Tsunoda, K.:
Motion by mean curvature from Glauber--Kawasaki dynamics.
\emph{J. Stat. Phys.} \textbf{177}(2), (2019), 183--208.


\bibitem{giardina2009duality}
Giardina, C. and Kurchan, J. and Redig, F. and Vafayi, K.:
Duality and hidden symmetries in interacting particle systems.
\emph{J. Stat. Phys.} \textbf{135}(1), (2009), 25--55.

\bibitem{Goncalves08}
Gon{\c{c}}alves, P.:
Central limit theorem for a tagged particle in asymmetric simple exclusion.
\emph{Stochastic Process. Appl.} \textbf{118}(3), (2008), 474--502.

\bibitem{jara2020stochastic}
Jara, M. and Landim, C.:
The stochastic heat equation as the limit of a stirring dynamics perturbed by a voter model,
\ARXIV{2008.03076}.

\bibitem{jaram18nonequilireaction}
Jara, M. and Menezes, O.:
Non-equilibrium fluctuations for a reaction-diffusion model via relative entropy,
\ARXIV{1810.03418}.

\bibitem{jara2018non}
Jara, M. and Menezes, O.:
Non-equilibrium fluctuations of interacting particle systems,
\ARXIV{1810.09526}.

\bibitem{JLT21}
Jara, M., Landim, C. and Tsunoda, K.:
Derivation of viscous Burgers equations from weakly asymmetric exclusion processes.
\emph{Ann. Inst. Henri Poincar{\'e} Probab. Statist.} \textbf{57}(1), (2021), 169--194.

\bibitem{klscaling}
Kipnis, C. and Landim, C.:
Scaling limits of interacting particle systems.
Volume 320 of Grundlehren der mathematischen wissenschaften.
\emph{Springer Science \& Business Media}, 2013.

\bibitem{Majda84}
Majda, A.:
Compressible fluid flow and systems of conservation laws in several space variables.
Volume 53 of Applied Mathematical Science.
\emph{Springer New York}, 1984.

\bibitem{MarcheO18}
Marchesani, S. and Olla, S.:
Hydrodynamic limit for an anharmonic chain under boundary tension.
\emph{Nonlinearity} \textbf{31}(11), (2018), 4979--5035.

\bibitem{malek1996}
M{\'a}rek, J., Ne{\v c}as, J., Rokyta, M. and R{\r u}{\v z}i{\v c}ka, M.:
Weak and Measure-valued Solutions to Evolutionary PDEs.
\emph{Chapman and Hall/CRC}, 1996.

\bibitem{OVY93}
Olla, S., Varadhan, S. R. S. and Yau, H. T.:
Hydrodynamical limit for a Hamiltonian system with weak noise.
\emph{Comm. Math. Phys.} \textbf{155}(3), (1993), 523--560.

\bibitem{OllaX20}
Olla, S. and Xu, L.:
Equilibrium fluctuation for an anharmonic chain with boundary conditions in the Euler scaling limit.
\emph{Nonlinearity} \textbf{33}(4), (2020), 1466--1498.

\bibitem{rezakhanlou91}
Rezakhanlou, F.:
Hydrodynamic limit for attractive particle systems on $\mathbb{Z}^{d}$.
\emph{Comm. Math. Phys.} \textbf{140}(3), (1991), 417--448.

\bibitem{seppalainen2001perturbation}
Sepp\"al\"ainen, T.:
Perturbation of the equilibrium for a totally asymmetric stick process in one dimension.
\emph{Ann. Probab.} \textbf{29}(1), (2001), 176--204.

\bibitem{toth2002between}
T{\'o}th, B.  and Valk{\'o}, B.:
Between equilibrium fluctuations and Eulerian scaling: perturbation of equilibrium for a class of deposition models.
\emph{J. Stat. Phys.} \textbf{109}(1), (2002), 177--205.

\bibitem{toth2005perturbation}
T{\'o}th, B. and Valk{\'o}, B.:
Perturbation of singular equilibria of hyperbolic two-component systems: a universal hydrodynamic limit.
\emph{Comm. Math. Phys.}, \textbf{256}(1), (2005), 111--157.

\bibitem{Valko06}
Valk\'o, B.:
Hydrodynamic limit for perturbation of a hyperbolic equilibrium point in two-component systems.
\emph{Ann. Inst. Henri Poincar{\'e} Probab. Statist.} \textbf{42}(1), (2006), 61--80.

\bibitem{Wu09}
Wu, L.
Gradient estimates of Poisson equations on Riemannian manifolds and applications.
\emph{J. Func. Anal.} \textbf{257}, (2009), 4015--4033.

\bibitem{Xu20}
Xu, L.:
Hyperbolic scaling limit of non-equilibrium fluctuations for a weakly anharmonic chain.
\emph{Electron. J. Probab.} \textbf{25}, (2020), 1--40.

\bibitem{yau1991relative}
Yau, H. T.:
Relative entropy and hydrodynamics of Ginzburg-Landau models.
\emph{Lett. Math. Phys.} \textbf{22}(1), (1991), 63--80.

\end{thebibliography}
\end{document}